\documentclass[11pt,reqno, a4paper]{amsart}
\usepackage{amssymb,amsmath,amsthm,amstext,amscd, url}
\usepackage{mathrsfs}
\newtheorem{assumptionA}{A-\hspace{-1.2mm}}
\theoremstyle{plain}
\newtheorem{lemma}{Lemma}
\newtheorem{theorem}{Theorem}
\newtheorem{remark}{Remark}
\newtheorem{cor}{Corollary}
\begin{document}
\title[Explicit Milstein-type Scheme]{On Milstein approximations with varying coefficients: the case of super-linear diffusion coefficients}

\author[C. Kumar]{Chaman Kumar}
\address[Chaman Kumar]{Stat-Math Unit, Indian Statistical Institute, Delhi}
\email{rkchaman@isid.ac.in}

% second author
\author[S. Sabanis]{Sotirios Sabanis}
\address[Sotirios Sabanis]{School of Mathematics, University of Edinburgh, Edinburgh}
\email{s.sabanis@ed.ac.uk}

\begin{abstract}

\bigskip
A new class of explicit Milstein schemes, which approximate stochastic differential equations (SDEs) with superlinearly growing drift and diffusion coefficients, is proposed in this article. It is shown, under very mild conditions, that these explicit schemes converge in $\mathcal L^p$ to the solution of the corresponding SDEs with optimal rate.

\noindent {\it AMS subject classifications}: Primary 60H35; secondary 65C30.
\end{abstract}

\maketitle

\section{Introduction}

Following the approach of \cite{sabanis2015}, we extend the techniques of constructing explicit approximations to the solutions of SDEs with super-linear coefficients in order to develop Milstein-type schemes with optimal rate of (strong) convergence.

Recent advances in the area of numerical approximations of such non-linear SDEs have produced new Euler-type schemes, e.g. see  \cite{hutzenthaler2012, sabanis2013, tretyakov2013, sabanis2015, hutzenthaler2015, dareiotis2015}, which are explicit in nature and hence computationally more efficient than their implicit counterparts. High-order schemes have also been developed in this direction. In particular, Milstein-type (order 1.0) schemes for SDEs with super-linear drift coefficients have been studied in \cite{wanga2013} and in \cite{kumar2015} with the latter article extending the results to include L\'evy noise, i.e. discontinuous paths. Furthermore, both drift and diffusion coefficients are allowed to grow super-linearly in \cite{zhang2014} and in \cite{Kruse et al.}. The latter reference has significantly relaxed the assumptions on the regularity of SDE coefficients by using the notions of C-stability and B-consistency. More precisely, the authors in \cite{Kruse et al.} produced optimal rate of convergence results in the case where the drift and diffusion coefficients are only (once) continuously differentiable functions. Our results, which were developed at around the same time as the latter reference by using different methodologies, are obtained under the same relaxed assumptions with regards to the regularity that is required of the SDE coefficients. Crucially, we relax further the moments bound requirement which is essential for practical applications.

We illustrate the above statement by considering an example which appears in \cite{Kruse et al.}, namely the one-dimensional SDE given by
\begin{align*}
dx_t=x_t(1-x_t^2)dt+\sigma(1-x_t^2)dw_t, \qquad \forall \, t \in [0,\,T],
\end{align*}
with initial value $x_0$ and a positive constant $\sigma$. Theorem \ref{thm:main} below yields that for $p_0=14$ (note that $\rho=2$) one obtains optimal rate of convergence in $\mathcal L^2$ (when $\sigma^2 \le \frac{2}{13}$ and $p_1>2$ such that $\sigma^2(p_1-1)\le 1$) whereas the corresponding result in \cite{Kruse et al.}, Table 1 in Section 8, requires  $p_0=18$ for their explicit (projective) scheme. The same requirement, i.e. $p_0=14$, as in this article is only achieved by the implicit schemes considered in \cite{Kruse et al.}.

Finally, we note that Theorem \ref{thm:main} establishes optimal rate of convergence results (under suitable assumptions) in $\mathcal L^p$ for $p>2$, which is, to the best of the authors' knowledge, the first such results in the case of SDEs with super-linear coefficients.

We conclude this section by introducing some notations which are used in this article. The Euclidean norm of a $d$-dimensional vector $b$ and the Hilbert-Schmidt norm of a $d\times m$ matrix $\sigma$ are denoted by $|b|$ and $|\sigma|$ respectively. The transpose of a matrix $\sigma$ is denoted  by $\sigma^*$. The $i$th element of  $b$ is denoted by $b^i$, whereas $\sigma^{(i,j)}$ and $\sigma^{(j)}$ stand for $(i,j)$-th element and $j$-th column of $\sigma$ respectively for every  $i=1,\ldots,d$ and $j=1,\ldots,m$. Further, $xy$ denotes the inner product of two $d$-dimensional vectors $x$ and $y$.  The notation $\lfloor a \rfloor$ stands for the integer part of a positive real number $a$. Let $D$ denote an operator such that for a function $f: \mathbb{R}^d \to \mathbb{R}^d$, $Df(.)$ gives a $d \times d$ matrix  whose $(i,j)$-th entry is $\frac{\partial f^i(.)}{\partial x^j}$ for every $i,j=1,\ldots,d$.  For every $j=1,\ldots,m$, let $\Lambda^j$ be an operator such that for a function $g: \mathbb{R}^d \to \mathbb{R}^{d \times m}$, $\Lambda^j g(.)$ gives a matrix of order $d \times m$ whose $(i,k)$-th entry is given by
$$
[\Lambda^j g(.)]_{(i,k)} := \sum_{u=1}^d g^{(u,j)}(.) \frac{\partial g^{(i,k)}(.)}{\partial x^u}
$$
for every $i=1,\ldots, d$, $k=1,\ldots,m$.

%---------------------------------------------------------------
\section{Main Results} \label{sec:main:results}
%---------------------------------------------------------------
Suppose $(\Omega, \{\mathscr{F}_t\}_{t \geq 0}, \mathscr{F}, P)$  is a complete filtered probability space satisfying the usual conditions, i.e. the filtration is  right continuous and $\mathscr{F}_0$ contains all $P$-null sets. Let $T>0$ be a fixed constant and  $(w_t)_{t \in [0, T]}$ denote an ${\mathbb{R}}^m-$valued  standard Wiener process. Further, suppose that $b(x)$ and $\sigma(x)$ are  $\mathscr{B}(\mathbb R^d)$-measurable functions with values in ${\mathbb{R}}^d$ and ${\mathbb{R}}^{d\times m}$ respectively. Moreover, $b(x)$ and $\sigma(x)$ are continuously differentiable in $x\in \mathbb{R}^d$. For the purpose of this article, the following $d$-dimensional SDE is considered,
\begin{align} \label{eq:sde}
x_t=& \xi + \int_{0}^t b(x_s)ds+\int_{0}^t \sigma(x_s)dw_s,
\end{align}
almost surely for any $t \in [0,T] $, where $\xi$ is an $\mathscr{F}_{0}$-measurable  random variable in $\mathbb{R}^d$.

Let $p_0, p_1 \geq 2$ and $\rho \geq 1$ (or $\rho=0$) are fixed constants. For the purpose of this article,  the following assumptions are made.

%\begin{assumptionA} \label{as:sde:coninuity}
%$b(x)$ is a continuous function in $x \in \mathbb{R}^d$.
%\end{assumptionA}

\begin{assumptionA} \label{as:sde:initial}
$E|\xi|^{p_0} < \infty$.
\end{assumptionA}

\begin{assumptionA} \label{as:sde:growth}
There exists a constant $L>0$ such that
$$
2xb(x)+ (p_0-1)|\sigma(x)|^2 \leq L(1+|x|^2)
$$
for any $x\in \mathbb{R}^d$.
\end{assumptionA}

\begin{assumptionA} \label{as:sde:lipschitz}
There exists a constant $L>0$ such that
$$
2(x-\bar{x})(b(x)-b(\bar{x}))+ (p_1-1)|\sigma(x)-\sigma(\bar{ x})|^2 \leq L|x-\bar{x}|^2
$$
for any $x, \bar{x} \in \mathbb{R}^d$.
\end{assumptionA}

\begin{assumptionA} \label{as:sde:lipschitz:b'}
There exists a constant $L>0$ such that
\begin{align*}
|D b(x)-D b(\bar{x})|  & \leq L (1+|x|+|\bar{x}|)^{\rho-1}|x-\bar{x}|
\end{align*}
for any $x, \bar{x} \in \mathbb{R}^d$.
\end{assumptionA}

\begin{assumptionA} \label{as:sde:lipschitz:sigma'}
There exists  a constant $L>0$ such that, for every $j=1,\ldots, m$,
\begin{align*}
|D\sigma^{(j)}(x)-D\sigma^{(j)}(\bar{x})|  & \leq L (1+|x|+|\bar{x}|)^{\frac{\rho-2}{2}}|x-\bar{x}|
\end{align*}
for any $x,\bar{x} \in \mathbb{R}^d$.
\end{assumptionA}

%\begin{assumptionA} \label{as:sde:sigma''}
%There exist  constants $L>0$ and $\rho\geq 0$ such that
%\begin{align*}
%\Big|\sum_{u_1,u_2=1}^d \sum_{l_1,l_2=1}^m \Lambda^{j_1} \sigma^{(u_1,l_1)}(x)\Lambda^{j_1} \sigma^{(u_2,l_1)}(x)\frac{\partial^2 \sigma(x)}{\partial x^{u_1} \partial x^{u_2}} \Big| \leq L(1+|x|^{\frac{5 \rho + 7}{2}})
%\end{align*}
%for any $x \in \mathbb{R}^d$.
%\end{assumptionA}

\begin{remark} \label{rem:poly:b}
Assumption A-\ref{as:sde:lipschitz:b'} means that there is  a constant $L>0$ such that
\begin{align*}
\Big|\frac{\partial b^i(x)}{\partial x^j}\Big| & \leq L(1+|x|)^{\rho}
\end{align*}
for any $x \in \mathbb{R}^d$ and for every $i,j=1,\ldots,d$. As a consequence, one also obtains that there exists a constant $L>0$ such that
\begin{align*}
|b(x)-b(\bar{x})| & \leq L (1+|x|+|\bar{x}|)^{\rho}|x-\bar{x}|
\end{align*}
for any $x , \bar{x} \in \mathbb{R}^d$. Moreover, this implies that $b(x)$ satisfies,
\begin{align*}
|b(x)|   & \leq L(1+|x|)^{\rho+1}
\end{align*}
for any $x \in \mathbb{R}^d$. Furthermore, due to Assumption A-\ref{as:sde:lipschitz:sigma'}, there exists a constant $L>0$ such that
\begin{align*}
\Big|\frac{\partial \sigma^{(i,j)}(x)}{\partial x^k}\Big|&\leq L(1+|x|)^{\frac{\rho}{2}}
\end{align*}
for any $x \in \mathbb{R}^d$ and for every $i,k=1,\ldots,d$, $j=1,\ldots,m$. Also, Assumption A-\ref{as:sde:lipschitz} implies
\begin{align*}
|\sigma(x)-\sigma(\bar{x})| & \leq L (1+|x|+|\bar{x}|)^\frac{\rho}{2} |x-\bar{x}| \notag
\end{align*}
for any $x, \bar{x} \in \mathbb{R}^d$.  Moreover, this means $\sigma(x)$ satisfies,
\begin{align*}
|\sigma(x)|&\leq L(1+|x|)^{\frac{\rho+2}{2}}
\end{align*}
for any $x  \in \mathbb{R}^d$. In addition, one notices that
\begin{align*}
|\Lambda^j \sigma(x)| &\leq L(1+|x|)^{\rho+1}
\end{align*}
for any $x \in \mathbb{R}^d$ and for every $j=1,\ldots,m$.
\end{remark}

For every $n \in \mathbb{N}$ and $x \in \mathbb{R}^d$, we define the following functions,
\begin{align*}
 b^n(x)& :=\frac{b(x)}{1+n^{-\theta}|x|^{2\rho \theta}},
 \\
\sigma^n(x)& :=  \frac{\sigma(x)}{1+n^{-\theta}|x|^{2\rho \theta}},
 \end{align*}
where $\theta \geq \frac{1}{2}$ and, similarly, for the purposes of establishing a new, explicit Milstein-type scheme, for every $j=1,\ldots,m$, we define
 \begin{align*}
\Lambda^{n,j}\sigma(x) & := \frac{ \Lambda^j \sigma(x)}{1+n^{-\theta}|x|^{2\rho \theta}}.
\end{align*}
\begin{remark} The case $\theta=1/2$ is studied in \cite{sabanis2015}, without the use of $\Lambda^{n,j}\sigma(x)$, as the aim is the formulation of a new  explicit Euler-type scheme. Throughout this article, $\theta$ is taken to be $1$, which corresponds to an order $1.0$ Milstein scheme. By taking different values of $\theta=1.5,2,2.5, \ldots$ and by appropriately controlling higher order terms, one can obtain optimal rate of convergence results for higher order schemes by adopting the approach developed in \cite{sabanis2015} and in this article.
\end{remark}

Moreover, let us also define
\begin{align*}
 \sigma_1^n(t, x)&:=\sum_{j=1}^{m}\int_{\kappa(n,t)}^t   \Lambda^{n,j}\sigma(x) dw_r^j =\sum_{j=1}^{m} \Lambda^{n,j}\sigma(x) (w_t^j-w_{\kappa(n,t)}^j) \notag
\end{align*}
and hence set
$$
\tilde \sigma^n(t,x):= \sigma^n(x)  + \sigma_1^n(t,x)
$$
almost surely for any $x \in \mathbb{R}^d$, $n \in \mathbb{N}$  and $t  \in [0,T] $.
\begin{remark} \label{rem:bn:sign:}
Due to Remark \ref{rem:poly:b}, one immediately notices that
\begin{align*}
|b^n(x)| &\leq \min( K n^{\frac{1}{2}}(1+|x|),|b(x)|)
\\
|\sigma^n(x)|^2 & \leq \min( K n^{\frac{1}{2}}(1+|x|^2),|\sigma(x)|^2)
\\
|\Lambda^{n,j}\sigma(x)| & \leq \min( K n^{\frac{1}{2}}(1+|x|),|\Lambda(x)|)
\end{align*}
for every $n \in \mathbb{N}, x \in \mathbb{R}^d$ and $j=1,\ldots,m$.
\end{remark}
Let us define $\kappa(n,t):=\lfloor nt \rfloor/ n$ for any $t \in [0,T]$. We propose below a new variant of the Milstein scheme with coefficients which vary according to the choice of the time step. The aim is to approximate solutions of non-linear SDEs such as equation \eqref{eq:sde}. The new explicit scheme is given below
\begin{align} \label{eq:milstein:diffusion}
x_t^n=& \xi + \int_{0}^t  b^n(x_{\kappa(n,s)}^n)ds+\int_{0}^t \tilde{\sigma}^n(s, x_{\kappa(n,s)}^n)dw_s
\end{align}
almost surely for any $t \in [0,T]$.
\begin{remark}
In the following, $K>0$ denotes a generic constant that varies from place to place, but is always independent of $n \in \mathbb{N}$.
\end{remark}

The main result of this article is stated in the following theorem.
\begin{theorem} \label{thm:main}
Let Assumptions A-\ref{as:sde:initial} to A-\ref{as:sde:lipschitz:sigma'} be satisfied with $p_0\ge 2(3\rho+1)$ and $p_1>2$. Then, the explicit Milstein-type scheme \eqref{eq:milstein:diffusion} converges in $\mathcal L^p$ to the true solution of SDE \eqref{eq:sde} with a rate of convergence equal to $1.0$, i.e. for every $n \in \mathbb{N}$
\begin{align} \label{rate-result}
\sup_{0 \leq t \leq T}E|x_t-x_t^n|^p \leq K n^{-p},
\end{align}
when $p=2$. Moreover, if $p_0\ge 4(3\rho+1)$, then \eqref{rate-result} is true for any $p \leq \frac{p_0}{3\rho+1}$ provided that $p<p_1$.
\end{theorem}

\begin{remark}
One observes immediately that for the case $\rho=0$, one recovers, due to Assumptions A-\ref{as:sde:initial} to A-\ref{as:sde:lipschitz:sigma'} and Theorem \ref{thm:main}, the classical Milstein framework and results (with some improvement perhaps as the coefficients of \eqref{eq:sde} are required only to be once continuously differentiable in this article).
\end{remark}

\begin{remark} \label{split}
In order to ease notation, it is chosen not to explicitly present the calculations for, and thus it is left as an exercise to the reader, the case where the drift and/or the diffusion coefficients contain parts which are Lipschitz continuous and grow at most linearly (in $x$). In such a case, the analysis for these parts follows closely the classical approach and the main theorem/results of this article remain true. Furthermore, note that such a statement applies also in the case of non-autonomous coefficients in which typical assumptions for the smoothness of coefficients in $t$  are considered (as, for example, in \cite{Kruse et al.}).
\end{remark}

The details of the proof of the main result, i.e. Theorem \ref{thm:main}, and of the required lemmas are given in the next two sections.
%------------------------------------------------------------------
\section{Moment Bounds}
%------------------------------------------------------------------
It is a well-known fact that due to Assumptions A-\ref{as:sde:initial} to A-\ref{as:sde:lipschitz}, the $p_0$-th moment of the true solution of \eqref{eq:sde} is bounded uniformly in time.
\begin{lemma} \label{lem:mb:true}
Let Assumptions A-\ref{as:sde:initial} to A-\ref{as:sde:lipschitz} be satisfied. Then, there exists a  unique solution $(x_t)_{t \in [0,T]}$ of SDE \eqref{eq:sde} and the following holds,
$$
\sup_{0 \leq t \leq T }E|x_t|^{p_0} \leq K.
$$
\end{lemma}
The proof of the above lemma can be found in many textbooks, e.g. see \cite{mao1997}. The following lemmas are required in order to allow one to obtain moment bounds for the new explicit scheme \eqref{eq:milstein:diffusion}.

\begin{remark}
Another useful observations is that for every fixed $n \in \mathbb{N}$ and due to Remark \ref{rem:bn:sign:}, the $p_0$-th moment of the new Milstein-type scheme \eqref{eq:milstein:diffusion} is bounded uniformly in time (as in the case of the classical Milstein scheme/framework with SDE coefficients which grow at most linearly). Clearly, one cannot claim at this point that such a bound is independent of $n$. However, the use of stopping times in the derivation of moment bounds henceforth can be avoided.
\end{remark}

%-------------sig1:no:rate----------------------------------------
\begin{lemma} \label{lem:Esign1}
%-----------------------------------------------------------------
Let Assumption A-\ref{as:sde:lipschitz:sigma'} be satisfied. Then,
$$
E|\sigma_1^n(t, x_{\kappa(n,t)}^n)|^{p_0} \leq K (1+ E|x_{\kappa(n,t)}^n|^{p_0})
$$
for  any $t \in [0,T]$ and $n \in \mathbb{N}$.
\end{lemma}
\begin{proof}
On using an elementary inequality of stochastic integrals and H\"older's inequality, one obtains
\begin{align*}
E|\sigma_1^{n}(t, x_{\kappa(n,t)}^n)|^{p_0} & = K E\Big|\sum_{j=1}^{m}\int_{\kappa(n,t)}^t \Lambda^{n,j} \sigma(x_{\kappa(n,s)}^n)dw_s^j\Big|^{p_0}
\\
%& \leq K \sum_{j=1}^{m} E \Big(\int_{\kappa(n,t)}^t |\Lambda^{n,j} \sigma(x_{\kappa(n,s)}^n)|^2  ds\Big)^{\frac{p_0}{2}}
%\\
& \leq K n^{-\frac{p_0}{2}+1} E \int_{\kappa(n,t)}^t |\Lambda^{n,j} \sigma(x_{\kappa(n,s)}^n)|^{p_0}  ds
\end{align*}
which due to Remark \ref{rem:bn:sign:} gives
\begin{align*}
E|\sigma_1^{n}(t, x_{\kappa(n,t)}^n)|^{p_0} & \leq K n^{-\frac{p_0}{2}+1} E \int_{\kappa(n,t)}^t n^{\frac{p_0}{2}}(1+|x_{\kappa(n,s)}^n|^{p_0})  ds
\end{align*}
and hence the proof completes.
\end{proof}
The following corollary is an immediate consequence of Lemma \ref{lem:Esign1} and Remark \ref{rem:bn:sign:}.
%-------------sig:no:rate----------------------------------------
\begin{cor} \label{lem:tilde:sign:no:rate}
%-----------------------------------------------------------------
Let Assumption A-\ref{as:sde:lipschitz:sigma'} be satisfied. Then
$$
E|\tilde{\sigma}^n(t, x_{\kappa(n,t)}^n)|^{p_0} \leq K n^{\frac{p_0}{4}}(1+E|x_{\kappa(n,t)}^n|^{p_0})
$$
for any $n \in \mathbb{N}$ and $t \in [0,T]$.
\end{cor}
When $p_0=2$, one proceeds with the following lemma (which is important for the case $\rho=0$).
%-----------------scheme:moment:bound------------------------------
\begin{lemma} \label{lem:mbound}
%------------------------------------------------------------------
Let Assumptions A-\ref{as:sde:initial} to A-\ref{as:sde:lipschitz:sigma'} be satisfied. Then, the explicit Milstein-type scheme \eqref{eq:milstein:diffusion} satisfies the following,
$$
\sup_{n \in \mathbb{N}} \sup_{0 \leq t \leq T}E|x_t^n|^{2} \leq K.
$$
\end{lemma}
%------------------------------------------------------------------
\begin{proof}
By It\^o's formula,  one obtains
\begin{align*}
|x_t^n|^{2} & =|\xi|^{2} + 2 \int_0^t  x_s^n b^n(x_{\kappa(n,s)}^n)ds + 2 \int_0^t  x_s^n \tilde{\sigma}^n(s, x_{\kappa(n,s)}^n) dw_s
\\
&\quad +  \int_0^t  |\tilde{\sigma}^{n}(s, x_{\kappa(n,s)}^n)|^2 ds
\end{align*}
for any $t \in [0,T]$. Also, one uses $|z_1+z_2|^2 = |z_1|^2 + 2\sum_{i=1}^{d}\sum_{j=1}^{m}z_1^{(i,j)}z_2^{(i,j)}+|z_2|^2$ for any $z_1, z_2 \in \mathbb{R}^{d \times m}$ to estimate the last term of the above equation,
\begin{align*}
& E|x_t^n|^{2}  =E|\xi|^{2} +  E\int_0^t  2(x_s^n-x_{\kappa(n,s)}^n) b^n(x_{\kappa(n,s)}^n) ds
\\
& +  E\int_0^t  \{2x_{\kappa(n,s)}^n b^n(x_{\kappa(n,s)}^n)+ |\sigma^n(x_{\kappa(n,s)}^n)|^2 \} ds + E\int_0^t |\sigma^n_1(s, x_{\kappa(n,s)}^n)|^2 ds
\\
&  + 2 E \sum_{i=1}^{d}\sum_{j=1}^{m}\int_0^t \sigma^{n,(i,j)}(x_{\kappa(n,s)}^n)\sigma^{n,(i,j)}_1(s, x_{\kappa(n,s)}^n) ds
\end{align*}
which further implies due to Lemma \ref{lem:Esign1} (with $p_0=2$),
\begin{align*}
& E|x_t^n|^{2}  \leq E|\xi|^{2} +  2 E\int_0^t  \int^{s}_{\kappa(n,s)} b^n(x_{\kappa(n,r)}^n) dr b^n(x_{\kappa(n,s)}^n) ds
\\
&+  2 E\int_0^t  \int^{s}_{\kappa(n,s)}\tilde{\sigma}^n(r,x_{\kappa(n,r)}^n) dw_r b^n(x_{\kappa(n,s)}^n) ds
\\
& +  E\int_0^t  \frac{2x_{\kappa(n,s)}^n b(x_{\kappa(n,s)}^n)+ |\sigma(x_{\kappa(n,s)}^n)|^2 }{1+n^{-1}|x_{\kappa(n,s)}^n|^{2\rho+4}} ds + KE\int_0^t (1+|x_{\kappa(n,s)}^n|^2) ds
\\
&  + 2 E \sum_{i=1}^{d}\sum_{j=1}^{m}\int_0^t \sigma^{n,(i,j)}(x_{\kappa(n,s)}^n)\sum_{k=1}^{m}\int^{s}_{\kappa(n,s)} \Lambda^{n,k}\sigma^{(i,j)}(x_{\kappa(n,r)}^n)dw_r^k  ds
\end{align*}
and then on the application of Assumption A-\ref{as:sde:growth}, Remark \ref{rem:bn:sign:} (also notice that third and last terms are zero) gives
\begin{align*}
& \sup_{0 \leq s \leq t}E|x_s^n|^{2}  \leq E|\xi|^{2} +  K + K\int_0^t \sup_{0 \leq r \leq s}E|x_r^n|^2 ds<\infty
\end{align*}
for any $t \in [0,T]$. The proof completes on using Gronwall's lemma.
\end{proof}
When $p_0 \geq 4$, one proceeds with the following lemma.
%----------------------------------------------
\begin{lemma} \label{lem:mbound}
%------------------------------------------------------------------
Let Assumptions A-\ref{as:sde:initial} to A-\ref{as:sde:lipschitz:sigma'} be satisfied. Then, the explicit Milstein-type scheme \eqref{eq:milstein:diffusion} satisfies the following,
$$
\sup_{n \in \mathbb{N}} \sup_{0 \leq t \leq T}E|x_t^n|^{p_0} \leq K.
$$
\end{lemma}
%----------------------------------------------------------
\begin{proof}
By It\^o's formula,  one obtains
\begin{align*}
|x_t^n|^{p_0} & =|\xi|^{p_0} + p_0 \int_0^t |x_s^n|^{p_0-2} x_s^n b^n(x_{\kappa(n,s)}^n)ds
\\
&\quad +p_0 \int_0^t |x_s^n|^{p_0-2} x_s^n \tilde{\sigma}^n(s, x_{\kappa(n,s)}^n) dw_s
\\
& \quad +\frac{p_0(p_0-2)}{2} \int_0^t |x_s^n|^{p_0-4} |\tilde{\sigma}^{n*}(s, x_{\kappa(n,s)}^n)x_s^n|^2 ds
\\
&\quad + \frac{p_0}{2} \int_0^t |x_s^n|^{p_0-2} |\tilde{\sigma}^{n}(s, x_{\kappa(n,s)}^n)|^2 ds,
\end{align*}
and then on taking expectation along with Schwarz inequality,
\begin{align*}
E|x_t^n|^{p_0} & \leq E|\xi|^{p_0}+p_0 E\int_0^t |x_s^n|^{p_0-2} (x_s^n-x_{\kappa(n,s)}^n) b^n(x_{\kappa(n,s)}^n)ds
\\
& \quad+p_0 E\int_0^t |x_s^n|^{p_0-2} x_{\kappa(n,s)}^n b^n(x_{\kappa(n,s)}^n)ds
\\
& \quad+\frac{p_0(p_0-1)}{2} E \int_0^t |x_s^n|^{p_0-2} |\tilde{\sigma}^{n}(s, x_{\kappa(n,s)}^n)|^2 ds
\end{align*}
for any $t \in [0,T]$. Then, one uses $|z_1+z_2|^2 = |z_1|^2+2\sum_{i=1}^{d}\sum_{j=1}^{m}z_1^{(i,j)}z_2^{(i,j)}+|z_2|^2$ for $z_1,z_2 \in \mathbb{R}^{d\times m}$ to obtain the following estimates,
\begin{align}
E|x_t^n&|^{p_0}  \leq E|\xi|^{p_0}+p_0 E\int_0^t |x_s^n|^{p_0-2} (x_s^n-x_{\kappa(n,s)}^n) b^n(x_{\kappa(n,s)}^n)ds \notag
\\
& +\frac{p_0}{2} E\int_0^t |x_s^n|^{p_0-2} \{2x_{\kappa(n,s)}^n b^n(x_{\kappa(n,s)}^n)+(p_0-1)|\sigma^{n}(x_{\kappa(n,s)}^n)|^2\}ds \notag
\\
& +\frac{p_0(p_0-1)}{2} E \int_0^t |x_s^n|^{p_0-2} |\sigma^{n}_1(s, x_{\kappa(n,s)}^n)|^2 ds \notag
\\
& +p_0(p_0-1) E \int_0^t |x_s^n|^{p_0-2} \sum_{i=1}^{d}\sum_{j=1}^{m}\sigma^{n,(i,j)}(x_{\kappa(n,s)}^n) \sigma^{n,(i,j)}_1(s, x_{\kappa(n,s)}^n) ds \notag
\\
 =: & C_1+C_2+C_3+C_4+C_5. \label{eq:C1+C4}
\end{align}
Here, $C_1:=E|\xi|^{p_0}$. In order to estimate $C_2$, one notices that it can be written as
\begin{align*}
C_2:=& p_0 E\int_0^t |x_s^n|^{p_0-2} (x_s^n-x_{\kappa(n,s)}^n) b^n(x_{\kappa(n,s)}^n)ds
\\
=& p_0 E\int_0^t |x_s^n|^{p_0-2} \int^s_{\kappa(n,s)}  b^n(x_{\kappa(n,r)}^n)dr b^n(x_{\kappa(n,s)}^n)ds
\\
&+ p_0 E\int_0^t |x_s^n|^{p_0-2} \int^s_{\kappa(n,s)}  \tilde{\sigma}^n(r, x_{\kappa(n,r)}^n)dw_r b^n(x_{\kappa(n,s)}^n)ds
\end{align*}
which on the application of Remark  \ref{rem:bn:sign:} and Young's inequality gives,
\begin{align*}
C_2  \leq & K \int_0^t E |x_s^n|^{p_0} ds  + K \int_0^t E |x_{\kappa(n,s)}^n|^{p_0} ds
\\
& + p_0 E\int_0^t |x_{\kappa(n,s)}^n|^{p_0-2} \int^s_{\kappa(n,s)}  \tilde{\sigma}^n(r, x_{\kappa(n,r)}^n)dw_r b^n(x_{\kappa(n,s)}^n)ds
\\
&+ p_0 E\int_0^t (|x_s^n|^{p_0-2}-|x_{\kappa(n,s)}^n|^{p_0-2}) \int^s_{\kappa(n,s)}  \tilde{\sigma}^n(r, x_{\kappa(n,r)}^n)dw_r b^n(x_{\kappa(n,s)}^n)ds
\end{align*}
for any $t \in [0,T]$. Further, one observes that the second term of the above equation is zero  and the third term can be estimated by the application of It\^o's formula as below,
\begin{align*}
C_2 & \leq K \int_0^t \sup_{0 \leq r \leq s}E|x_r^n|^{p_0} ds + K E\int_0^t \int_{\kappa(n,s)}^s |x_r^n|^{p_0-4} x_r^n b^n(x_{\kappa(n,r)}^n)dr
\\
& \qquad \times \int^s_{\kappa(n,s)}  \tilde{\sigma}^n(r, x_{\kappa(n,r)}^n)dw_r b^n(x_{\kappa(n,s)}^n)ds
\\
&+ K E\int_0^t \int_{\kappa(n,s)}^s |x_r^n|^{p_0-4} x_r^n \tilde{\sigma}^n(r, x_{\kappa(n,r)}^n) dw_r
\\
& \qquad \times \int^s_{\kappa(n,s)}  \tilde{\sigma}^n(r, x_{\kappa(n,r)}^n)dw_r b^n(x_{\kappa(n,s)}^n)ds
\\
&+ K E\int_0^t \int_{\kappa(n,s)}^s |x_r^n|^{p_0-4} |\tilde{\sigma}^{n}(r, x_{\kappa(n,r)}^n)|^2 dr
\\
& \qquad \times \mid \int^s_{\kappa(n,s)}  \tilde{\sigma}^n(r, x_{\kappa(n,r)}^n)dw_r \mid \mid b^n(x_{\kappa(n,s)}^n)\mid ds
\end{align*}
for any $t \in [0,T]$. Due to Remark \ref{rem:bn:sign:} along with an elementary inequality of stochastic integrals, the following estimates can be obtained,
\begin{align*}
C_2 & \leq K \int_0^t \sup_{0 \leq r \leq s}E|x_r^n|^{p_0} ds
\\
& + K n E\int_0^t  \int_{\kappa(n,s)}^s (1+|x_{\kappa(n,s)}^n|^2)|x_r^n|^{p_0-3}  dr \Big|\int^s_{\kappa(n,s)}  \tilde{\sigma}^n(r, x_{\kappa(n,r)}^n)dw_r \Big| ds
\\
&+ K n^{\frac{1}{2}}E\int_0^t \int_{\kappa(n,s)}^s (1+|x_{\kappa(n,s)}^n|)|x_r^n|^{p_0-3} |\tilde{\sigma}^n(r, x_{\kappa(n,r)}^n)|^2 dr ds
\\
&+ K n^\frac{1}{2} E\int_0^t \int_{\kappa(n,s)}^s (1+|x_{\kappa(n,s)}^n|) |x_r^n|^{p_0-4} |\tilde{\sigma}^{n}(r,x_{\kappa(n,r)}^n)|^2  dr
\\
& \qquad \times \Big|\int^s_{\kappa(n,s)}   \tilde{\sigma}^n(r, x_{\kappa(n,r)}^n)dw_r \Big|ds
\end{align*}
which can also be estimated as,
\begin{align*}
C_2 & \leq  K \int_0^t \sup_{0 \leq r \leq s} E|x_r^n|^{p_0} ds
\\
&+ K  E\int_0^t  n^\frac{3}{4} \int_{\kappa(n,s)}^s (1+|x_{\kappa(n,s)}^n|^2)|x_r^n|^{p_0-3}  dr n^\frac{1}{4}\Big|\int^s_{\kappa(n,s)}  \tilde{\sigma}^n(r, x_{\kappa(n,r)}^n)dw_r \Big| ds
\\
&+ K E\int_0^t  \int_{\kappa(n,s)}^s n^{1-\frac{2}{p_0}}(1+|x_{\kappa(n,s)}^n|)|x_r^n|^{p_0-3} n^{-\frac{1}{2}+\frac{2}{p_0}}|\tilde{\sigma}^n(r, x_{\kappa(n,r)}^n)|^2 dr ds
\\
&+ K  E\int_0^t n^\frac{1}{4} \int_{\kappa(n,s)}^s (1+|x_{\kappa(n,s)}^n|) |x_r^n|^{p_0-4}  |\tilde{\sigma}^{n}(r, x_{\kappa(n,r)}^n)|^2 dr
\\
&\qquad \times n^{\frac{1}{4}} \Big|\int^s_{\kappa(n,s)}   \tilde{\sigma}^n(r, x_{\kappa(n,r)}^n)dw_r \Big|ds
\end{align*}
and then one uses Young's inequality to obtain the following estimates,
\begin{align*}
C_2 & \leq K \int_0^t \sup_{0 \leq r \leq s}E|x_r^n|^{p_0} ds
\\
& + K n^\frac{3p_0}{4(p_0-1)} E\int_0^t  \Big(\int_{\kappa(n,s)}^s (1+|x_{\kappa(n,s)}^n|^2)|x_r^n|^{p_0-3}  dr\Big)^\frac{p_0}{p_0-1} ds
\\
&+ K n E\int_0^t \int_{\kappa(n,s)}^s \big((1+|x_{\kappa(n,s)}^n|) |x_r^n|^{p_0-3} \big)^\frac{p_0}{p_0-2} dr ds
\\
&+ K  n^\frac{p_0}{4(p_0-1)} E\int_0^t \Big(\int_{\kappa(n,s)}^s (1+|x_{\kappa(n,s)}^n|)|x_r^n|^{p_0-4} \mid  \tilde{\sigma}^n(r,x_{\kappa(n,r)}^n) \mid^2 dr\Big)^\frac{p_0}{p_0-1}ds
\\
&+ K n^{\frac{p_0}{4}} E\int_0^t \Big|\int^s_{\kappa(n,s)}  \tilde{\sigma}^n(r,x_{\kappa(n,r)}^n)dw_r \Big|^{p_0} ds
\\
& + K n^{-\frac{p_0}{4}+1} E\int_0^t \int^s_{\kappa(n,s)}  |\tilde{\sigma}^n(r, x_{\kappa(n,r)}^n)|^{p_0}dr  ds
\end{align*}
for any $t \in [0,T]$. Further, by the application of H\"older's inequality and an elementary inequality of stochastic integrals, one obtains the following estimates,
\begin{align*}
C_2 & \leq K \int_0^t \sup_{0 \leq r \leq s}E|x_r^n|^{p_0} ds
\\
& + K n^{\frac{3p_0}{4(p_0-1)}-\frac{p_0}{p_0-1}+1} E\int_0^t  \int_{\kappa(n,s)}^s (1+|x_{\kappa(n,s)}^n|^\frac{2p_0}{p_0-1})|x_r^n|^{\frac{p_0(p_0-3)}{p_0-1}}   dr ds
\\
&+ K nE\int_0^t \int_{\kappa(n,s)}^s (1+|x_{\kappa(n,s)}^n|^\frac{p_0}{p_0-2}) |x_r^n|^{\frac{(p_0-3)p_0}{p_0-2}}  dr ds
\\
&+ K  n^{\frac{p_0}{4(p_0-1)}-\frac{p_0}{p_0-1}+1} E\int_0^t \int_{\kappa(n,s)}^s (1+|x_{\kappa(n,s)}^n|^\frac{p_0}{p_0-1})|x_r^n|^{\frac{p_0(p_0-4)}{p_0-1}}
\\
& \qquad \times \mid  \tilde{\sigma}^n(r,x_{\kappa(n,r)}^n)\mid^\frac{2p_0}{p_0-1} drds
\\
&+ K n^{-\frac{p_0}{4}+1} E\int_0^t \int^s_{\kappa(n,s)}  |\tilde{\sigma}^n(r, x_{\kappa(n,r)}^n)|^{p_0}dr  ds
%\\
%& + K n^{-\frac{p_0}{4}+1} E\int_0^t \int^s_{\kappa(n,s)}  |\tilde{\sigma}^n(r, x_{\kappa(n,r)}^n)|^{p_0}dr  ds
\end{align*}
which due to Corollary \ref{lem:tilde:sign:no:rate} yields
\begin{align*}
C_2 & \leq K \int_0^t \sup_{0 \leq r \leq s}E|x_r^n|^{p_0} ds
\\
& + K E\int_0^t   (1+|x_{\kappa(n,s)}^n|^\frac{2p_0}{p_0-1}) n^{-\frac{p_0}{4(p_0-1)}+1} \int_{\kappa(n,s)}^s |x_r^n|^{\frac{p_0(p_0-3)}{p_0-1}}   dr ds
\\
&+ K E\int_0^t (1+|x_{\kappa(n,s)}^n|^\frac{p_0}{p_0-2}) n \int_{\kappa(n,s)}^s  |x_r^n|^{\frac{(p_0-3)p_0}{p_0-2}}  dr ds
\\
&+ K E\int_0^t (1+|x_{\kappa(n,s)}^n|^\frac{p_0}{p_0-1}) n^{-\frac{3p_0}{4(p_0-1)}+1}\int_{\kappa(n,s)}^s |x_r^n|^{\frac{p_0(p_0-4)}{p_0-1}}
\\
&\qquad \times \mid  \tilde{\sigma}^n(r,x_{\kappa(n,r)}^n)\mid^\frac{2p_0}{p_0-1}  drds
\\
&+ K n^{-\frac{p_0}{4}+1} \int_0^t \int^s_{\kappa(n,s)}  n^\frac{p_0}{4}(1+E|x_{\kappa(n,r)}^n|^{p_0}) dr  ds
\end{align*}
and then on further application of Young's inequality, following estimates are obtained
\begin{align*}
C_2 & \leq K+ K \int_0^t \sup_{0 \leq r \leq s}E|x_r^n|^{p_0} ds + K E\int_0^t   (1+|x_{\kappa(n,s)}^n|^{p_0}) ds
\\
&  + K E\int_0^t n^{-\frac{p_0}{4(p_0-3)}+\frac{p_0-1}{p_0-3}} \Big(\int_{\kappa(n,s)}^s |x_r^n|^{\frac{p_0(p_0-3)}{p_0-1}}   dr\Big)^\frac{p_0-1}{p_0-3} ds
%\\
%&+ K E\int_0^t (1+|x_{\kappa(n,s)}^n|)^{p_0}ds
\\
& + K E\int_0^t n^\frac{p_0-2}{p_0-3} \Big(\int_{\kappa(n,s)}^s  |x_r^n|^{\frac{(p_0-3)p_0}{p_0-2}}  dr\Big)^\frac{p_0-2}{p_0-3} ds
%\\
%&+ K   E\int_0^t (1+|x_{\kappa(n,s)}^n|)^{p_0} ds
\\
& + K   E\int_0^t  n^{-\frac{3p_0}{4(p_0-2)}+\frac{p_0-1}{p_0-2}} \Big(\int_{\kappa(n,s)}^s |x_r^n|^{\frac{p_0(p_0-4)}{p_0-1}} \mid  \tilde{\sigma}^n(r,x_{\kappa(n,r)}^n)\mid^\frac{2p_0}{p_0-1} dr\Big)^\frac{p_0-1}{p_0-2} ds
\end{align*}
for any $t \in [0,T]$. Thus, on using H\"older's inequality, one obtains
\begin{align*}
C_2 & \leq K+ K \int_0^t \sup_{0 \leq r \leq s}E|x_r^n|^{p_0} ds   + K E\int_0^t n^{-\frac{p_0}{4(p_0-3)}+1} \int_{\kappa(n,s)}^s |x_r^n|^{p_0}   dr ds
\\
& + K E\int_0^t n \int_{\kappa(n,s)}^s  |x_r^n|^{p_0}  dr ds
\\
&+ K   E\int_0^t  n^{-\frac{3p_0}{4(p_0-2)}+1} \int_{\kappa(n,s)}^s |x_r^n|^{\frac{p_0(p_0-4)}{p_0-2}}  \mid  \tilde{\sigma}^n(r,x_{\kappa(n,r)}^n)\mid^\frac{2p_0}{p_0-2} dr ds
\end{align*}
for any $t \in [0,T]$. Also, one can write above inequality as
\begin{align*}
C_2 & \leq K+ K \int_0^t \sup_{0 \leq r \leq s}E|x_r^n|^{p_0} ds
\\
&+ K   E\int_0^t   \int_{\kappa(n,s)}^s  n^{\frac{p_0-4}{p_0-2}} |x_r^n|^{\frac{p_0(p_0-4)}{p_0-2}}  n^{\frac{-7p_0+16}{4(p_0-2)}+1}\mid  \tilde{\sigma}^n(r,x_{\kappa(n,r)}^n)\mid^\frac{2p_0}{p_0-2} dr ds
\end{align*}
which on using Young's inequality yields,
\begin{align*}
C_2 & \leq K+ K \int_0^t \sup_{0 \leq r \leq s}E|x_r^n|^{p_0} ds
\\
&+ K   E\int_0^t   \int_{\kappa(n,s)}^s  n |x_r^n|^{p_0} dr ds + K   E\int_0^t   \int_{\kappa(n,s)}^s n^{-\frac{3p_0-8}{8}}\mid  \tilde{\sigma}^n(r,x_{\kappa(n,r)}^n)\mid^{p_0} dr ds
\end{align*}
and due to Corollary \ref{lem:tilde:sign:no:rate}, one obtains
\begin{align*}
C_2 & \leq K+ K \int_0^t \sup_{0 \leq r \leq s}E|x_r^n|^{p_0} ds   + K  n^{-\frac{p_0}{8}} \int_0^t   (1+ E\mid x_{\kappa(n,s)}^n\mid^{p_0}) ds
\end{align*}
and hence finally the following estimates are obtained,
\begin{align}
C_2 &\leq  K + K \int_0^t \sup_{0 \leq r \leq s}E|x_r^n|^{p_0} ds \label{eq:C2}
\end{align}
for any $t \in [0,T]$. For $C_3$, one uses Assumption A-\ref{as:sde:growth} to obtain the following,
\begin{align}
C_3 &:=\frac{p_0}{2} E\int_0^t |x_s^n|^{p_0-2} \{2x_{\kappa(n,s)}^n b^n(x_{\kappa(n,s)}^n)+(p_0-1)|\sigma^{n}(x_{\kappa(n,s)}^n)|^2\}ds \notag
\\
& = \frac{p_0}{2} E\int_0^t |x_s^n|^{p_0-2} \frac{2x_{\kappa(n,s)}^n b(x_{\kappa(n,s)}^n)+(p_0-1)|\sigma(x_{\kappa(n,s)}^n)|^2}{1+n^{-1}|x_{\kappa(n,s)}^n|^{2\rho+4}} ds \notag
\\
& \leq K E\int_0^t |x_s^n|^{p_0-2} \frac{1+|x_{\kappa(n,s)}^n|^2}{1+n^{-1}|x_{\kappa(n,s)}^n|^{2\rho+4}} ds \notag
\end{align}
which due to Young's inequality gives,
\begin{align}
C_3\leq  K + K \int_0^t \sup_{0 \leq r \leq s}E|x_r^n|^{p_0} ds \label{eq:C3}
\end{align}
for any $t \in [0,T]$. Furthermore, by using Young's inequality, $C_4$ in \eqref{eq:C1+C4} is estimated  as,
\begin{align*}
C_4&:=\frac{p_0(p_0-1)}{2} E \int_0^t |x_s^n|^{p_0-2} |\sigma^{n}_1(s, x_{\kappa(n,s)}^n)|^2 ds
\\
&\leq K E \int_0^t |x_s^n|^{p_0} ds + K E \int_0^t  |\sigma^{n}_1(s, x_{\kappa(n,s)}^n)|^{p_0} ds
\end{align*}
and then on the application of Lemma \ref{lem:Esign1}, one obtains
\begin{align}
C_4  \leq K + K \int_0^t \sup_{0 \leq r \leq s}E|x_r^n|^{p_0} ds  \label{eq:C4}
\end{align}
for any $t \in [0,T]$. Now, for estimating $C_5$, one writes
\begin{align*}
C_5&:=p_0(p_0-1) E \int_0^t |x_s^n|^{p_0-2} \sum_{i=1}^{d}\sum_{j=1}^{m}\sigma^{n,(i,j)}(x_{\kappa(n,s)}^n) \sigma^{n,(i,j)}_1(s, x_{\kappa(n,s)}^n) ds
\\
& \leq K E \int_0^t |x_{\kappa(n,s)}^n|^{p_0-2} \sum_{i=1}^{d}\sum_{j=1}^{m}\sigma^{n,(i,j)}(x_{\kappa(n,s)}^n)\sum_{k=1}^{m} \int_{\kappa(n,s)}^{s} \Lambda^{n,k} \sigma^{n,(i,j)}(x_{\kappa(n,r)}^n)dw_r^k ds
\\
& +K E \int_0^t (|x_s^n|^{p_0-2}-|x_{\kappa(n,s)}^n|^{p_0-2})  \sum_{i=1}^{d}\sum_{j=1}^{m}\sigma^{n,(i,j)}(x_{\kappa(n,s)}^n)\sigma^{n,(i,j)}_1(s, x_{\kappa(n,s)}^n) ds
\end{align*}
for any $t \in [0,T]$. Clearly, the first term is zero and one uses It\^o's formula for the second term to obtain the following,
\begin{align*}
C_5 &\leq K E \int_0^t \int_{\kappa(n,s)}^{s} |x_r^n|^{p_0-4}x_r^nb^n(x_{\kappa(n,r)}^n)dr
\\
& \qquad \times \sum_{i=1}^{d}\sum_{j=1}^{m}\sigma^{n,(i,j)}(x_{\kappa(n,s)}^n)\sigma^{n,(i,j)}_1(s, x_{\kappa(n,s)}^n) ds
\\
&+K E \int_0^t \int_{\kappa(n,s)}^{s} |x_r^n|^{p_0-4}x_r^n \tilde \sigma^n(r,x_{\kappa(n,r)}^n)dw_r
\\
& \qquad \times \sum_{i=1}^{d}\sum_{j=1}^{m}\sigma^{n,(i,j)}(x_{\kappa(n,s)}^n)\sum_{k=1}^{m} \int_{\kappa(n,s)}^{s} \Lambda^{n,k} \sigma^{n,(i,j)}(x_{\kappa(n,r)}^n)dw_r^k ds
\\
&+K E \int_0^t \int_{\kappa(n,s)}^{s} |x_r^n|^{p_0-4} |\tilde \sigma^n(r,x_{\kappa(n,r)}^n)|^2dr
\\
& \qquad \times |\sum_{i=1}^{d}\sum_{j=1}^{m}\sigma^{n,(i,j)}(x_{\kappa(n,s)}^n)\sigma^{n,(i,j)}_1(s, x_{\kappa(n,s)}^n)| ds
\end{align*}
which on using Schwarz inequality and Remark \ref{rem:bn:sign:}  along with an elementary inequality of stochastic integrals yields,
\begin{align*}
C_5 &\leq K n^{\frac{3}{4}} E \int_0^t (1+|x_{\kappa(n,s)}^n|^2)\int_{\kappa(n,s)}^{s} |x_r^n|^{p_0-3}dr  |\sigma^{n}_1(s,x_{\kappa(n,s)}^n)| ds
\\
&+K n^{\frac{3}{4}}  E \int_0^t (1+|x_{\kappa(n,s)}^n|^2) \int_{\kappa(n,s)}^{s} |x_r^n|^{p_0-3} |\tilde \sigma^n(r,x_{\kappa(n,r)}^n)|dr ds
\\
&+K n^{\frac{1}{4}}E \int_0^t (1+|x_{\kappa(n,s)}^n|) \int_{\kappa(n,s)}^{s} |x_r^n|^{p_0-4} |\tilde \sigma^n(r,x_{\kappa(n,r)}^n)|^2dr|\sigma^{n}_1(s,x_{\kappa(n,s)}^n)| ds
\end{align*}
for any $t \in [0,T]$. Further, one uses Young's inequality  to obtain the following estimates,
\begin{align*}
C_5 &\leq K+ K   \int_0^t E|x_{\kappa(n,s)}^n|^{p_0} ds
\\
& + n^{\frac{3p_0}{4(p_0-2)}}E\int_0^t\big(\int_{\kappa(n,s)}^{s} |x_r^n|^{p_0-3}dr \big)^\frac{p_0}{p_0-2} |\sigma^{n}_1(s,x_{\kappa(n,s)}^n)|^\frac{p_0}{p_0-2} ds
\\
&+K n^{\frac{3p_0}{4(p_0-2)}}  E \int_0^t  \big(\int_{\kappa(n,s)}^{s} |x_r^n|^{p_0-3} |\tilde \sigma^n(r,x_{\kappa(n,r)}^n)|dr\big)^\frac{p_0}{p_0-2} ds
\\
&+K n^{\frac{p_0}{4(p_0-1)}}E \int_0^t  \big(\int_{\kappa(n,s)}^{s} |x_r^n|^{p_0-4} |\tilde \sigma^n(r,x_{\kappa(n,r)}^n)|^2dr\big)^\frac{p_0}{p_0-1}|\sigma^{n}_1(s,x_{\kappa(n,s)}^n)|^\frac{p_0}{p_0-1} ds
\end{align*}
which on the application of H\"older's inequality yields
\begin{align*}
C_5 &\leq K+ K   \int_0^t E|x_{\kappa(n,s)}^n|^{p_0} ds
\\
& + n^{\frac{3p_0}{4(p_0-2)}-\frac{p_0}{p_0-2}+1} E\int_0^t\int_{\kappa(n,s)}^{s} |x_r^n|^{\frac{p_0(p_0-3)}{p_0-2}}dr |\sigma^{n}_1(s,x_{\kappa(n,s)}^n)|^\frac{p_0}{p_0-2} ds
\\
&+K n^{\frac{3p_0}{4(p_0-2)}-\frac{p_0}{p_0-2}+1}  E \int_0^t \int_{\kappa(n,s)}^{s} |x_r^n|^{\frac{(p_0-3)p_0}{p_0-2}} |\tilde \sigma^n(r,x_{\kappa(n,r)}^n)|^\frac{p_0}{p_0-2}dr ds
\\
&+K n^{\frac{p_0}{4(p_0-1)}-\frac{p_0}{p_0-1}+1}E \int_0^t \int_{\kappa(n,s)}^{s} |x_r^n|^{\frac{(p_0-4)p_0}{p_0-1}} |\tilde \sigma^n(r,x_{\kappa(n,r)}^n)|^\frac{2p_0}{p_0-1}dr|\sigma^{n}_1(s,x_{\kappa(n,s)}^n)|^\frac{p_0}{p_0-1} ds
\end{align*}
and then again using Young's inequality, the following estimates are obtained,
\begin{align*}
C_5 &\leq K+ K   \int_0^t E|x_{\kappa(n,s)}^n|^{p_0} ds +KE\int_0^t |\sigma^{n}_1(s,x_{\kappa(n,s)}^n)|^{p_0} ds
\\
& + n^{\frac{3p_0}{4(p_0-3)}-\frac{p_0}{p_0-3}+\frac{p_0-2}{p_0-3}} E\int_0^t\big(\int_{\kappa(n,s)}^{s} |x_r^n|^{\frac{p_0(p_0-3)}{p_0-2}}dr\big)^\frac{p_0-2}{p_0-3} ds
\\
&+K  E \int_0^t \int_{\kappa(n,s)}^{s} n^{\frac{p_0-3}{p_0-2}} |x_r^n|^{\frac{(p_0-3)p_0}{p_0-2}} n^{\frac{3p_0}{4(p_0-2)}-\frac{p_0}{p_0-2}+1-\frac{p_0-3}{p_0-2}} |\tilde \sigma^n(r,x_{\kappa(n,r)}^n)|^\frac{p_0}{p_0-2}dr ds
\\
&+K n^{\frac{p_0}{4(p_0-2)}-\frac{p_0}{p_0-2}+\frac{p_0-1}{p_0-2}}E \int_0^t \big(\int_{\kappa(n,s)}^{s} |x_r^n|^{\frac{(p_0-4)p_0}{p_0-1}} |\tilde \sigma^n(r,x_{\kappa(n,r)}^n)|^\frac{2p_0}{p_0-1}dr\big)^\frac{p_0-1}{p_0-2} ds
\end{align*}
for any $t \in [0,T]$. The application of Young's inequality, H\"older's inequality  and Lemma \ref{lem:Esign1} implies
\begin{align*}
C_5 &\leq K+ K   \int_0^t \sup_{0 \leq r \leq s}E|x_r^n|^{p_0} ds + n^{-\frac{p_0}{4(p_0-3)}+1} E\int_0^t\int_{\kappa(n,s)}^{s} |x_r^n|^{p_0}dr ds
\\
&+K  E \int_0^t \int_{\kappa(n,s)}^{s} n |x_r^n|^{p_0} ds+K  E \int_0^t \int_{\kappa(n,s)}^{s} n^{-\frac{p_0}{4}+1} |\tilde \sigma^n(r,x_{\kappa(n,r)}^n)|^{p_0}dr ds
\\
&+K E \int_0^t \int_{\kappa(n,s)}^{s} n^{\frac{p_0-4}{p_0-2}} |x_r^n|^{\frac{(p_0-4)p_0}{p_0-2}} n^{\frac{p_0}{4(p_0-2)}-\frac{p_0}{p_0-2}+1-\frac{p_0-4}{p_0-2}}|\tilde \sigma^n(r,x_{\kappa(n,r)}^n)|^\frac{2p_0}{p_0-2}dr ds
\end{align*}
and once again using Young's inequality along with Corollary \ref{lem:tilde:sign:no:rate}, one obtains the following,
\begin{align*}
C_5 &\leq K+ K   \int_0^t \sup_{0 \leq r \leq s}E|x_r^n|^{p_0} d+K  E \int_0^t \int_{\kappa(n,s)}^{s} n^{-\frac{p_0}{4}+1} |\tilde \sigma^n(r,x_{\kappa(n,r)}^n)|^{p_0}dr ds
\\
&+K E \int_0^t \int_{\kappa(n,s)}^{s} n^{-\frac{3p_0}{8}+1}|\tilde \sigma^n(r,x_{\kappa(n,r)}^n)|^{p_0} dr ds
\end{align*}
for any $t \in [0,T]$. Thus, Corollary \ref{lem:tilde:sign:no:rate} gives
\begin{align} \label{eq:C5}
C_5 &\leq K+ K   \int_0^t \sup_{0 \leq r \leq s}E|x_r^n|^{p_0} ds
\end{align}
for any $t \in [0,T]$. By substituting estimates from \eqref{eq:C2}, \eqref{eq:C3}, \eqref{eq:C4}  and \eqref{eq:C5} in \eqref{eq:C1+C4}, the following estimates are obtained,
\begin{align}
\sup_{0 \leq s \leq t} E|x_s^n|^{p_0} \leq K + K \int_0^t \sup_{0 \leq r \leq s}E|x_r^n|^{p_0} ds<\infty  \notag
\end{align}
for any $t \in [0,T]$. The proof is completed by the Gronwall's lemma.
\end{proof}
%-----------------------------------------------------------------
\section{Proof of Main Result}
%-----------------------------------------------------------------
A simple application of the mean value theorem, which appears in the Lemma below, allows us to simplify substantially the proof of Theorem \ref{thm:main}. Furthermore, throughout this section, it is assumed that $p_0\ge 2(3\rho +1)$ and $p_1>2$.
\begin{lemma} \label{lem:goodlemma}
Let $f:\mathbb{R}^d \to \mathbb{R}$ be a continuously differentiable function which satisfies the following,
\begin{align} \label{eq:goodlemma}
|Df(x)-Df(\bar x)| \leq (1+|x|+|\bar x|)^{\gamma}|x-\bar x|
\end{align}
for all $x ,\bar x\in \mathbb{R}^d$ and for a fixed $\gamma \in \mathbb{R} $. Then, there exists a constant $L$ such that
$$
|f(x)-f(\bar x)-\sum_{i=1}^d\frac{\partial f(\bar x)}{\partial y^i}(x^i-\bar x^i)| \leq L(1+|x|+|\bar x|)^{\gamma}|x-\bar x|^2
$$
for any $x,\bar x \in \mathbb{R}^d$.
\end{lemma}
\begin{proof} By mean value theorem,
$$
f(x)-f(\bar x)=\sum_{i=1}^{d}\frac{f(qx+(1-q)\bar x)}{\partial y^i}(x^i-\bar x^i)
$$
for some $q\in (0,1)$. Hence, for a fixed $q \in (0,1)$,
\begin{align*}
|f(x)&-f(\bar x)-\sum_{i=1}^{d}\frac{\partial f(\bar x)}{\partial y^i}(x^i-\bar x^i)|
\\
&= \Big|\sum_{i=1}^{d}\frac{\partial f(qx+(1-q)\bar x)}{\partial y^i}(x^i-\bar x^i)-\sum_{i=1}^{d}\frac{\partial f(x)}{\partial y^i}(x^i-\bar x^i)|
\\
&\leq \sum_{i=1}^{d}\Big| \frac{\partial f(qx+(1-q)\bar x)}{\partial y^i}-\frac{\partial f(x)}{\partial y^i}\Big||x^i-\bar x^i|
\end{align*}
which on using equation \eqref{eq:goodlemma} completes the proof.
\end{proof}
\begin{lemma} \label{lem:sign1:rate}
Let Assumptions A-\ref{as:sde:initial} to A-\ref{as:sde:lipschitz:sigma'} be satisfied. Then, for every $n\in \mathbb{N} $
$$
\sup_{0 \leq t \leq T }E|\sigma_1^n(t, x_{\kappa(n,t)}^n)|^{p} \leq K n^{-\frac{p}{2}}
$$
for any $p \leq \frac{p_0}{\rho+1}$.
\end{lemma}
\begin{proof}
By the application of an elementary inequality of stochastic integrals, H\"older's inequality and Remarks [\ref{rem:poly:b}, \ref{rem:bn:sign:}], one obtains
\begin{align*}
E|\sigma_1^n(t, x_{\kappa(n,t)}^n)|^{p} & \leq K \sum_{j=1}^{m}E\Big|\int_{\kappa(n,t)}^t \Lambda^{n,j}\sigma(x_{\kappa(n,s)}^n)dw_s^j\Big|^{p}
\\
& \leq K n^{-\frac{p}{2}+1} E \int_{\kappa(n,t)}^t |\Lambda^{j}\sigma(x_{\kappa(n,s)}^n)|^{p} ds
\\
& \leq K n^{-\frac{p}{2}+1} E \int_{\kappa(n,t)}^t (1+|x_{\kappa(n,s)}^n|)^{(\rho + 1)p}  ds
\end{align*}
which due to Lemma \ref{lem:mbound} completes the proof.
\end{proof}
As a consequence of  the above lemma, one obtains the following corollary.
%----------------------------------------------
\begin{cor} \label{cor:tilde:sig}
%---------------------------------------------
Let Assumptions A-\ref{as:sde:initial} to A-\ref{as:sde:lipschitz:sigma'} be satisfied. Then, for every $n \in \mathbb{N}$,
$$
\sup_{0 \leq t \leq T }E|\tilde \sigma^n(t, x_{\kappa(n,t)}^n)|^{p} \leq K
$$
for any $p \leq \frac{p_0}{\rho+1}$.
\end{cor}
%---------------------one step------------------------------------
\begin{lemma} \label{lem:one:step}
%-----------------------------------------------------------------
Let Assumptions A-\ref{as:sde:initial} to A-\ref{as:sde:lipschitz:sigma'} be satisfied. Then, for every $n \in \mathbb{N}$,
$$
\sup_{0 \leq t \leq T}E|x_t^n-x_{\kappa(n,t)}^n|^{p} \leq K n^{-\frac{p}{2}}
$$
for any $p \leq \frac{p_0}{\rho +1}$.
\end{lemma}
\begin{proof}
Due to the scheme \eqref{eq:milstein:diffusion},
\begin{align}
E|x_t^n-x_{\kappa(n,t)}^n|^{p} \leq  & K E\Big|\int_{\kappa(n,t)}^t  b^n(x_{\kappa(n,s)}^n)ds \Big|^{p}  + K E\Big|\int_{\kappa(n,t)}^t \tilde{\sigma}^n(s, x_{\kappa(n,s)}^n) dw_s\Big|^{p} \notag
\end{align}
and then the application of H\"older's inequality along with an elementary inequality of stochastic integrals gives
\begin{align*}
E|x_t^n-x_{\kappa(n,t)}^n|^{p} \leq  & K n^{-p+1}E\int_{\kappa(n,t)}^t  |b^n(x_{\kappa(n,s)}^n)|^{p} ds
\\
& + K n^{-\frac{p}{2}+1} E\int_{\kappa(n,t)}^t|\tilde{\sigma}^n(s,x_{\kappa(n,s)}^n)|^{p} ds
\end{align*}
which on using Remarks [\ref{rem:poly:b}, \ref{rem:bn:sign:}] yields the following estimates,
\begin{align*}
E|x_t^n-x_{\kappa(n,t)}^n|^{p} \leq  & K n^{-p+1}E\int_{\kappa(n,t)}^t  (1+|x_{\kappa(n,s)}^n|)^{(\rho+1)p} ds
\\
& + K n^{-\frac{p}{2}+1} E\int_{\kappa(n,t)}^t|\tilde{\sigma}^n(s,x_{\kappa(n,s)}^n)|^{p} ds
\end{align*}
for any $t \in [0,T]$. Thus, one uses Lemma \ref{lem:mbound} and Corollary \ref{cor:tilde:sig} to complete the proof.
\end{proof}
%-----------------------b-bn:rate----------------------------------
\begin{lemma} \label{lem:b-bn:rate}
%------------------------------------------------------------------
Let Assumptions A-\ref{as:sde:initial} to A-\ref{as:sde:lipschitz:sigma'} be satisfied. Then, for every $n \in \mathbb{N}$,
\begin{align*}
\sup_{0 \leq t \leq T}E|b(x_{\kappa(n,t)}^n)-b^n(x_{\kappa(n,t)}^n)|^p \leq Kn^{-p}
\end{align*}
for any $p \leq \frac{p_0}{3 \rho+1}$.
\end{lemma}
\begin{proof}
One observes that
\begin{align*}
|b(x_{\kappa(n,t)}^n)-b^n(x_{\kappa(n,t)}^n)|&=n^{-1}\frac{|b(x_{\kappa(n,t)}^n)||x_{\kappa(n,t)}^n|^{2\rho}}{1+n^{-1}|x_{\kappa(n,t)}^n|^{2\rho}}  \leq n^{-1}(1+|x_{\kappa(n,t)}^n|)^{3\rho+1}
\end{align*}
and hence Lemma \ref{lem:mbound} completes the proof.
\end{proof}
%-----------------------b-bn:rate----------------------------------
\begin{lemma} \label{lem:si-sin:rate}
%------------------------------------------------------------------
Let Assumptions A-\ref{as:sde:initial} to A-\ref{as:sde:lipschitz:sigma'} be satisfied. Then, for  for every $n \in \mathbb{N}$,
\begin{align*}
\sup_{0 \leq t \leq T}E|\sigma(x_{\kappa(n,t)}^n)-\sigma^n(x_{\kappa(n,t)}^n)|^p \leq Kn^{-p}
\end{align*}
for any $p \leq \frac{p_0}{2.5 \rho+1}$.
\end{lemma}
\begin{proof}
The proof follows using same arguments as used in Lemma \ref{lem:b-bn:rate}.
\end{proof}
%---------------------sig-sig:rate--------------------------------
\begin{lemma} \label{lem:sig-sig:rate}
%-----------------------------------------------------------------
Let Assumptions A-\ref{as:sde:initial} to  A-\ref{as:sde:lipschitz:sigma'} be satisfied. Then, for every $n \in \mathbb{N}$,
$$
\sup_{0 \leq t \leq T}E|\sigma(x_t^n)-\sigma(x_{\kappa(n,t)}^n)-\sigma_1^n(t, x_{\kappa(n,t)}^n)|^p \leq K n^{-p}
$$
for any $p \leq \frac{p_0}{3\rho+1}$.
\end{lemma}
\begin{proof}
First, one observes that
\begin{align} \label{eq:sig1:g}
\sum_{u=1}^{d}&\frac{\partial \sigma^{(k,v)}(x_{\kappa(n,t)}^n)}{\partial x^u}(x_t^{n,u}-x_{\kappa(n,t)}^{n,u}) =\sum_{u=1}^{d}\frac{\partial \sigma^{(k,v)}(x_{\kappa(n,t)}^n)}{\partial x^u} \bigg( \int_{\kappa(n,t)}^{t}b^{n,u}(x_{\kappa(n,s)}^n)ds \notag
\\
&\qquad+ \int_{\kappa(n,t)}^{t}\sum_{j=1}^{m}\sigma^{n,(u,j)}(x_{\kappa(n,s)}^n)dw_s^j + \int_{\kappa(n,t)}^{t}\sum_{j=1}^{m}\sigma_1^{n,(u,j)}(s,x_{\kappa(n,s)}^n)dw_s^j \bigg) \notag
\\
&=\sum_{u=1}^{d}\frac{\partial \sigma^{(k,v)}(x_{\kappa(n,t)}^n)}{\partial x^u} \int_{\kappa(n,t)}^{t}b^{n,u}(x_{\kappa(n,s)}^n)ds+\sigma_1^{n,(k,v)}(t,x_{\kappa(n,t)}^n) \notag
\\
&\qquad+\sum_{u=1}^{d}\frac{\partial \sigma^{(k,v)}(x_{\kappa(n,t)}^n)}{\partial x^u} \int_{\kappa(n,t)}^{t}\sum_{j=1}^{m}\sigma_1^{n,(u,j)}(s,x_{\kappa(n,s)}^n)dw_s^j
\end{align}
for any $t \in [0,T]$. Also, one can write the following,
\begin{align*}
\sigma^{(k,v)}&(x_t^n)-\sigma^{(k,v)}(x_{\kappa(n,t)}^n)-\sigma_1^{n,(k,v)}(t, x_{\kappa(n,t)}^n)
\\
&=\sigma^{(k,v)}(x_t^n)-\sigma^{(k,v)}(x_{\kappa(n,t)}^n)-\sum_{u=1}^{d}\frac{\partial \sigma^{(k,v)}(x_{\kappa(n,t)}^n)}{\partial x^u}(x_t^{n,u}-x_{\kappa(n,t)}^{n,u})
\\
&\qquad+\sum_{u=1}^{d}\frac{\partial \sigma^{(k,v)}(x_{\kappa(n,t)}^n)}{\partial x^u}(x_t^{n,u}-x_{\kappa(n,t)}^{n,u})- \sigma_1^{n,(k,v)}(t, x_{\kappa(n,t)}^n)
\end{align*}
 and hence due to equation \eqref{eq:sig1:g}, Remark \ref{rem:poly:b} and Lemma \ref{lem:goodlemma} (with $\gamma=(\rho-2)/2$), one obtains
\begin{align*}
|\sigma^{(k,v)}&(x_t^n)-\sigma^{(k,v)}(x_{\kappa(n,t)}^n)-\sigma_1^{n,(k,v)}(t, x_{\kappa(n,t)}^n)|
\\
&\leq L(1+|x_t^n|+|x_{\kappa(n,t)}^n|)^{\frac{\rho-2}{2}}|x_t^{n}-x_{\kappa(n,t)}^{n}|^2
\\
&\qquad+\Big|\sum_{u=1}^{d}\frac{\partial \sigma^{(k,v)}(x_{\kappa(n,t)}^n)}{\partial x^u}\int_{\kappa(n,t)}^t b_s^{n,u} (x_{\kappa(n,s)}^n)ds \Big|
\\
&+\Big|\sum_{u=1}^{d}\frac{\partial \sigma^{(k,v)}(x_{\kappa(n,t)}^n)}{\partial x^u} \int_{\kappa(n,t)}^{t}\sum_{j=1}^{m}\sigma_1^{n,(u,j)}(s,x_{\kappa(n,s)}^n)dw_s^j\Big|
\end{align*}
for any $t \in [0,T]$. Thus, on the application of H\"older's inequality and an elementary inequality of stochastic integrals along with Remarks [\ref{rem:poly:b}, \ref{rem:bn:sign:}], the following estimates are obtained,
\begin{align*}
E|\sigma^{(k,v)}&(x_t^n)-\sigma^{(k,v)}(x_{\kappa(n,t)}^n)-\sigma_1^{n,(k,v)}(t, x_{\kappa(n,t)}^n)|^p
\\
&\leq KE(1+|x_t^n|+|x_{\kappa(n,t)}^n|)^{\frac{\rho p}{2}}|x_t^{n}-x_{\kappa(n,t)}^{n}|^{2p}
\\
&\qquad+Kn^{-p}E(1+|x_{\kappa(n,t)}^n|)^{\frac{\rho p}{2}+(\rho+1)p}
\\
&+Kn^{-\frac{p}{2}+1} \int_{\kappa(n,t)}^{t}E(1+|x_{\kappa(n,t)}^n|)^\frac{\rho p}{2}|\sigma_1^{n}(s,x_{\kappa(n,s)}^n)|^pds
\end{align*}
for any $t\in[0,T]$. One again uses H\"older's inequality  and obtains,
\begin{align*}
E&|\sigma^{(k,v)}(x_t^n)-\sigma^{(k,v)}(x_{\kappa(n,t)}^n)-\sigma_1^{n,(k,v)}(t, x_{\kappa(n,t)}^n)|^p
\\
&\leq K\{E(1+|x_t^n|+|x_{\kappa(n,t)}^n|)^\frac{p_0}{2}\}^{\frac{\rho p}{p_0}}\{E|x_t^{n}-x_{\kappa(n,t)}^{n}|^\frac{2pp_0}{p_0-\rho p}\}^\frac{p_0-\rho p}{p_0}+ Kn^{-p}
\\
&+Kn^{-\frac{p}{2}+1} \int_{\kappa(n,t)}^{t}\{E(1+|x_{\kappa(n,t)}^n|)^{p_0}\}^\frac{\rho p}{2p_0}\{E|\sigma_1^{n}(s,x_{\kappa(n,s)}^n)|^\frac{2pp_0}{2p_0-\rho p}\}^\frac{2p_0-\rho p}{2p_0}ds
\end{align*}
for any $t \in [0,T]$. The proof is completed by Lemmas [\ref{lem:mbound}, \ref{lem:one:step}, \ref{lem:sign1:rate}].
\end{proof}

Let us at this point introduce $e_t^n:=x_t-x_{t}^n$ for any $t \in [0,T]$.
%---------------------------------------------------------
\begin{lemma} \label{lem:a-tilde a:rate:new}
%----------------------------------------------------------
Let Assumptions A-\ref{as:sde:initial} to  A-\ref{as:sde:lipschitz:sigma'} be satisfied. Then, for every $n \in \mathbb{N}$  and $t \in [0,T]$,
\begin{equation} \label{diff:b}
E\int_{0}^{t} |e_s^n|^{p-2}e_s^n(b(x_s^n)-b(x_{\kappa(n,s)}^n))ds\leq K  \int_{0}^t \sup_{0 \leq r \leq s}E|e_r^n|^p   ds + K n^{-p}
\end{equation}
for $p=2$. Furthermore, if $p_0 \ge 4(3\rho+1)$, then \eqref{diff:b} holds for any $p \leq \frac{p_0}{3\rho+1}$.
\end{lemma}
\begin{proof}
First, one writes the following,
\begin{align}
&E\int_{0}^t |e_s^n|^{p-2}  e_s^n(b(x_s^n)-b(x_{\kappa(n,s)}^n))ds  \notag
\\
& = \sum_{k=1}^d E\int_{0}^t |e_s^n|^{p-2} e_s^{n,k}(b^k(x_s^n)-b^k(x_{\kappa(n,s)}^n))ds \notag
\\
& = \sum_{k=1}^d E\int_{0}^t |e_s^n|^{p-2} e_s^{n,k}(b^k(x_s^n)-b^k(x_{\kappa(n,s)}^n)-\sum_{i=1}^{d}\frac{\partial b^k(x_{\kappa(n,s)}^n)}{\partial x^i}(x_s^{n,i}-x_{\kappa(n,s)}^{n,i}))ds \notag
\\
&\qquad+\sum_{k=1}^d E\int_{0}^t |e_s^n|^{p-2} e_s^{n,k} \sum_{i=1}^{d}\frac{\partial b^k(x_{\kappa(n,s)}^n)}{\partial x^i}(x_s^{n,i}-x_{\kappa(n,s)}^{n,i}) \notag
\\
& \leq \sum_{k=1}^d E\int_{0}^t |e_s^n|^{p-1} |b^k(x_s^n)-b^k(x_{\kappa(n,s)}^n)-\sum_{i=1}^{d}\frac{\partial b^k(x_{\kappa(n,s)}^n)}{\partial x^i}(x_s^{n,i}-x_{\kappa(n,s)}^{n,i})|ds \notag
\\
&\qquad+\sum_{k=1}^d E\int_{0}^t |e_s^n|^{p-2} e_s^{n,k} \sum_{i=1}^{d}\frac{\partial b^k(x_{\kappa(n,s)}^n)}{\partial x^i}\int_{\kappa(n,s)}^{s} b^{n,i}(x_{\kappa(n,r)}^n)dr \notag
\\
&\qquad+\sum_{k=1}^d E\int_{0}^t |e_s^n|^{p-2} e_s^{n,k} \sum_{i=1}^{d}\frac{\partial b^k(x_{\kappa(n,s)}^n)}{\partial x^i}\int_{\kappa(n,s)}^{s} \sum_{l=1}^{m}\bar \sigma^{n,(i,l)}(x_{\kappa(n,r)}^n)dw_r^l \notag
\\
&=: T_1+T_2+T_3 \label{eq:T1+T2+T3}
\end{align}
for any $t \in [0,T]$. Notice that when $p=2$, $|e_s^n|^{p-2}$ does not appear in $T_2$ and $T_3$ of the above equation.  One can keep note of this in mind in the following calculations because  their estimations require less computational efforts as compared to the case of $p \geq 4$.

$T_1$ can be estimated by using Lemma \ref{lem:goodlemma} (with $\gamma=\rho-1$) as below,
\begin{align*}
T_1&:=\sum_{k=1}^d E\int_{0}^t |e_s^n|^{p-1} |b^k(x_s^n)-b^k(x_{\kappa(n,s)}^n)-\sum_{i=1}^{d}\frac{\partial b^k(x_{\kappa(n,s)}^n)}{\partial x^i}(x_s^{n,i}-x_{\kappa(n,s)}^{n,i})|ds
\\
& \leq E\int_{0}^t |e_s^n|^{p-1} (1+|x_s^n|+|x_{\kappa(n,s)}^n|)^{\rho-1}|x_s^{n}-x_{\kappa(n,s)}^{n}|^2 ds \notag
\end{align*}
which on the application of Young's inequality and H\"older's inequality gives
\begin{align*}
T_1 &\leq \int_{0}^t E|e_s^n|^{p} ds + \int_{0}^t E(1+|x_s^n|+|x_{\kappa(n,s)}^n|)^{(\rho-1)p}|x_s^{n}-x_{\kappa(n,s)}^{n}|^{2p} ds \notag
\\
&\leq \int_{0}^t E|e_s^n|^{p} ds + \int_{0}^t \{E(1+|x_s^n|+|x_{\kappa(n,s)}^n|)^{p_0}\}^\frac{(\rho-1)p}{p_0}
\\
& \qquad \times \{E|x_s^{n}-x_{\kappa(n,s)}^{n}|^\frac{2pp_0}{p_0-(\rho-1)p}\}^\frac{p_0-(\rho-1)p}{p_0} ds \notag
\end{align*}
and then by using Lemmas [\ref{lem:mbound}, \ref{lem:one:step}], one obtains
\begin{align} \label{eq:T1}
T_1 \leq Kn^{-p}+\int_{0}^t \sup_{0 \leq r \leq s}E|e_r^n|^{p} ds
\end{align}
for any $t \in [0,T]$.

For $T_2$, one uses Schwarz, Young's and H\"older's inequalities and obtains the following estimates,
\begin{align}
T_2 &:=  \sum_{k=1}^d  E \int_{0}^t |e_s^n|^{p-2}e_s^{n,k} \sum_{i=1}^d\frac{\partial b^k(x_{\kappa(n,r)}^n)}{\partial x^i} \int_{\kappa(n,s)}^s  b^{n,i}(x_{\kappa(n,r)}^n)dr ds \notag
\\
& \leq  K  E \int_{0}^t  |e_s^{n}|^p ds +K n^{-p+1}\sum_{k,i=1}^d  E \int_{0}^t  \int_{\kappa(n,s)}^s |\frac{\partial b^k(x_{\kappa(n,r)}^n)}{\partial x^i}|^p |b^{n,i}(x_{\kappa(n,r)}^n)|^p dr ds \notag
\end{align}
which on the application of Remarks [\ref{rem:poly:b}, \ref{rem:bn:sign:}] yields
\begin{align*}
T_2 \leq K   \int_{0}^t  E|e_s^{n}|^p ds +K n^{-p+1}  E \int_{0}^t  \int_{\kappa(n,s)}^s (1+|x_{\kappa(n,r)}^n|)^{(2\rho+1)p} dr ds \notag
\end{align*}
for any $t \in [0,T]$. Furthermore, due to Lemma \ref{lem:mbound}, the following estimates are obtained,
\begin{align}
T_2 &\leq K   n^{-p} + \int_{0}^t  \sup_{0 \leq r \leq s}E |e_r^{n}|^p ds \label{eq:T2}
\end{align}
for any $t \in [0,T]$.

One can now proceed to the estimation of $T_3$. For this, one uses It\^o's formula and obtains the following estimates,
\begin{align} %\label{eq:ito:new1}
|e_s^n|^{p-2} & e_s^{n,k}=|e_{\kappa(n,s)}^n|^{p-2} e_{\kappa(n,s)}^{n,k} + \int_{\kappa(n,s)}^s |e_r^n|^{p-2}(b^k(x_r)-b^{n,k}(x_{\kappa(n,r)}^n) ) dr \notag
\\
&+  \int_{\kappa(n,s)}^s |e_r^n|^{p-2} \sum_{j=1}^m\big(\sigma^{(k,j)}(x_r)-\tilde{\sigma}^{n,(k,j)}(r, x_{\kappa(n,r)}^n)\big)dw_r^j \notag
\\
& + (p-2)  \int_{\kappa(n,s)}^s e_r^{n,k} |e_r^n|^{p-4} e_r^n (b(x_r)-b^n(x_{\kappa(n,r)}^n)) dr \notag
\\
&+ (p-2) \int_{\kappa(n,s)}^s e_r^{n,k} |e_r^n|^{p-4} e_r^n (\sigma(x_r)-\tilde{\sigma}^n(r, x_{\kappa(n,r)}^n))  dw_r  \notag
\\
& + \frac{(p-2)(p-4)}{2}  \int_{\kappa(n,s)}^s e_r^{n,k} |e_r^n|^{p-6}| (\sigma(x_r)-\tilde{\sigma}^n(r, x_{\kappa(n,r)}^n))^{*}e_r^n|^2 dr  \notag
\\
& + \frac{p-2}{2}  \int_{\kappa(n,s)}^s e_r^{n,k}|e_r^n|^{p-4}|\sigma(x_r)-\tilde{\sigma}^n(r, x_{\kappa(n,r)}^n)|^2 dr \notag
\\
& + (p-2) \int_{\kappa(n,s)}^s \sum_{j=1}^m (\sigma^{(k,j)}(x_r)-\tilde{\sigma}^{n,(k,j)}(r, x_{\kappa(n,r)}^n))|e_r^n|^{p-4} \notag
\\
& \qquad \times\sum_{u=1}^d e_r^{n,u} (\sigma^{(u,j)}(x_r)-\tilde{\sigma}^{n,(u,j)}(r, x_{\kappa(n,r)}^n)) dr \notag
\end{align}
for any $s \in [0,T]$. In the above equation, notice that when $p=2$, the last five terms are zero, $|e_{\kappa(n,s)}^n|^{p-2}$ is absent from the first term and $|e_r^n|^{p-2}$ does not appear in the second and third terms. This on substituting  in $T_3$ and then using Schwarz inequality gives the following estimates,
\begin{align}
T_3 & \leq  \sum_{k=1}^d E\int_{0}^t |e_{\kappa(n,s)}^n|^{p-2} e_{\kappa(n,s)}^{n,k} \int_{\kappa(n,s)}^s \sum_{l=1}^m \sum_{i=1}^d \frac{\partial b^k(x_{\kappa(n,r)}^n)}{\partial x^i}  \tilde{\sigma}^{n,(i,l)}(r, x_{\kappa(n,r)}^n)dw_r^l ds \notag
\\
&+\sum_{k=1}^d E\int_{0}^t \int_{\kappa(n,s)}^s |e_r^n|^{p-2}|b(x_r)- b^{n}(x_{\kappa(n,r)}^n) | dr \notag
\\
& \qquad\times  \Big|\int_{\kappa(n,s)}^s \sum_{l=1}^m \sum_{i=1}^d \frac{\partial b^k(x_{\kappa(n,r)}^n)}{\partial x^i}  \tilde{\sigma}^{n,(i,l)}(r, x_{\kappa(n,r)}^n)dw_r^l \Big|ds \notag
\\
& + \sum_{k=1}^d E\int_{0}^t \int_{\kappa(n,s)}^s |e_r^n|^{p-2} \sum_{j=1}^m\big(\sigma^{(k,j)}(x_r)-\tilde{\sigma}^{n,(k,j)}(r, x_{\kappa(n,r)}^n)\big)dw_r^j \notag
\\
& \qquad\times \int_{\kappa(n,s)}^s \sum_{l=1}^m \sum_{i=1}^d\frac{\partial b^k(x_{\kappa(n,r)}^n)}{\partial x^i}  \tilde{\sigma}^{n,(i,l)}(r, x_{\kappa(n,r)}^n)dw_r^l ds \notag
\\
& + K \sum_{k=1}^d E\int_{0}^t  \int_{\kappa(n,s)}^s e_r^{n,k} |e_r^n|^{p-4} e_r^n (\sigma(x_r)-\tilde{\sigma}^n(r, x_{\kappa(n,r)}^n))  dw_r \notag
\\
& \qquad \times \int_{\kappa(n,s)}^s \sum_{l=1}^m \sum_{i=1}^d \frac{\partial b^k(x_{\kappa(n,r)}^n)}{\partial x^i}  \tilde{\sigma}^{n,(i,l)}(r, x_{\kappa(n,r)}^n)dw_r^l  ds \notag
\\
& + K \sum_{k=1}^d E\int_{0}^t   \int_{\kappa(n,s)}^s |e_r^n|^{p-3}|\sigma(x_r)-\tilde{\sigma}^n(r, x_{\kappa(n,r)}^n)|^2 dr \notag
\\
& \qquad \times \Big|\int_{\kappa(n,s)}^s \sum_{l=1}^m \sum_{i=1}^d \frac{\partial b^k(x_{\kappa(n,r)}^n)}{\partial x^i}  \tilde{\sigma}^{n,(i,l)}(r, x_{\kappa(n,r)}^n)dw_r^l \Big| ds \notag
\\
&=:T_{31}+T_{32}+T_{33}+T_{34}+T_{35}  \label{eq:T31+T35}
\end{align}
for any $t \in [0,T]$. In order to estimate $T_{31}$, one writes,
\begin{align*}
T_{31}&:=K\sum_{k=1}^d   E \int_{0}^t |e_{\kappa(n,s)}^n|^{p-2} e_{\kappa(n,s)}^{n,k}
\\
& \qquad \times \int_{\kappa(n,s)}^s\sum_{l=1}^{m} \sum_{i=1}^d\frac{\partial b^k(x_{\kappa(n,r)}^n)}{\partial x^i}  \tilde{\sigma}^{n,(i,l)}(r, x_{\kappa(n,r)}^n)dw_r^l ds \notag
\\
& =K\sum_{k=1}^d   E \int_{0}^t |e_{\kappa(n,s)}^n|^{p-2} e_{\kappa(n,s)}^{n,k}
\\
& \qquad \times \int_{\kappa(n,s)}^s\sum_{l=1}^{m} \sum_{i=1}^d\frac{\partial b^k(x_{\kappa(n,r)}^n)}{\partial x^i}  \sigma^{n,(i,l)}(x_{\kappa(n,r)}^n)dw_r^l ds\notag
\\
& +K\sum_{k=1}^d  E \int_{0}^t |e_{\kappa(n,s)}^n|^{p-2} e_{\kappa(n,s)}^{n,k}
\\
& \qquad \times \int_{\kappa(n,s)}^s  \sum_{l=1}^{m} \sum_{i=1}^d\frac{\partial b^k(x_{\kappa(n,r)}^n)}{\partial x^i}  \sigma_1^{n,(i,l)}(r, x_{\kappa(n,r)}^n)dw_r^l ds\notag
\end{align*}
for any $t \in [0,T]$. In the above, notice that first term is zero. Then, on using the Young's inequality, H\"older's inequality and an elementary inequality of stochastic integrals, one obtains,
\begin{align}
T_{31} & \leq K  \int_{0}^t \sup_{0 \leq r \leq s}E|e_r^n|^p ds\notag
\\
& +K\sum_{k=1}^d   E \int_{0}^t \Big|\int_{\kappa(n,s)}^s\sum_{l=1}^{m} \sum_{i=1}^d \frac{\partial b^k(x_{\kappa(n,r)}^n)}{\partial x^i}  \sigma_1^{n,(i,l)}(r, x_{\kappa(n,r)}^n)dw_r^l\Big|^p ds\notag
\\
& \leq K  \int_{0}^t \sup_{0 \leq r \leq s}E|e_r^n|^p ds +Kn^{-\frac{p}{2}+1} E \int_{0}^t \int_{\kappa(n,s)}^s(1+|x_{\kappa(n,r)}^n|)^{\rho p} \notag
\\
&\qquad \times  |\sigma_1^n(r, x_{\kappa(n,r)}^n)|^p dr ds\notag
\\
& \leq K  \int_{0}^t \sup_{0 \leq r \leq s}E|e_r^n|^p ds +Kn^{-\frac{p}{2}+1}  \int_{0}^t \int_{\kappa(n,s)}^s \{E(1+|x_{\kappa(n,r)}^n|)^{p_0}\}^\frac{\rho p}{p_0}  \notag
\\
& \qquad \times \{E|\sigma_1(x_{\kappa(n,r)}^n)|^{\frac{p p_0}{p_0-\rho p}}\}^\frac{p_0-\rho p}{p_0} dr ds\notag
\end{align}
and then by using Lemmas [\ref{lem:mbound}, \ref{lem:sign1:rate}], one obtains
\begin{align}
T_{31}&\leq Kn^{-p}+K  \int_{0}^t \sup_{0 \leq r \leq s}E|e_r^n|^p ds\label{eq:T31}
\end{align}
for any $t \in [0,T]$. Moreover, for estimating $T_{32}$, one uses the following splitting,
\begin{align}
b(x_r)-\tilde{b}^{n}(x_{\kappa(n,r)}^n) = & (b(x_r)-b(x_r^n))+(b(x_r^n)-b(x_{\kappa(n,r)}^n)) \notag
\\
& +(b(x_{\kappa(n,r)}^n) -b^{n}(x_{\kappa(n,r)}^n)) \label{eq:b:split}
\end{align}
and hence $T_{32}$ can be estimated by
\begin{align*}
&T_{32} :=\sum_{k=1}^d E\int_{0}^t \int_{\kappa(n,s)}^s |e_r^n|^{p-2}|b(x_r)-\tilde{b}^{n}(x_{\kappa(n,r)}^n) | dr
\\
& \qquad \times \Big|\int_{\kappa(n,s)}^s \sum_{l=1}^m \sum_{i=1}^d \frac{\partial b^k(x_{\kappa(n,r)}^n)}{\partial x^i}  \tilde{\sigma}^{n,(i,l)}(r, x_{\kappa(n,r)}^n)dw_r^l \Big|ds
\\
& \leq K \sum_{k=1}^d E\int_{0}^t \int_{\kappa(n,s)}^s \big(n^\frac{1}{p}|e_r^n|\big)^{p-1}(1+|x_r|+|x_r^n|)^{\rho} dr
\\
& \qquad \times n^{-\frac{p-1}{p}} \Big|\int_{\kappa(n,s)}^s \sum_{l=1}^m \sum_{i=1}^d\frac{\partial b^k(x_{\kappa(n,r)}^n)}{\partial x^i}  \tilde{\sigma}^{n,(i,l)}(r, x_{\kappa(n,r)}^n)dw_r^l \Big|ds
\\
& + K \sum_{k=1}^d E\int_{0}^t \int_{\kappa(n,s)}^s \big(n^\frac{1}{p}|e_r^n|\big)^{p-2}(1+|x_r^n|+|x_{\kappa(n,r)}^n|)^{\rho}|x_r^n-x_{\kappa(n,r)}^n| dr
\\
 & \qquad\times n^{-\frac{p-2}{p}} \Big|\int_{\kappa(n,s)}^s\sum_{l=1}^m \sum_{i=1}^d\frac{\partial b^k(x_{\kappa(n,r)}^n)}{\partial x^i}  \tilde{\sigma}^{n,(i,l)}(r, x_{\kappa(n,r)}^n)dw_r^l \Big|ds
 \\
 & + K \sum_{k=1}^d E\int_{0}^t \int_{\kappa(n,s)}^s \big(n^\frac{1}{p}|e_r^n|\big)^{p-2}|b(x_{\kappa(n,r)}^n)-b^{n}(x_{\kappa(n,r)}^n) | dr
 \\
 & \qquad\times n^{-\frac{p-2}{p}} \Big|\int_{\kappa(n,s)}^s \sum_{l=1}^m \sum_{i=1}^d \frac{\partial b^k(x_{\kappa(n,r)}^n)}{\partial x^i}  \tilde{\sigma}^{n,(i,l)}(r, x_{\kappa(n,r)}^n)dw_r^l \Big|ds
\end{align*}
which on the application of Young's inequality gives
\begin{align*}
& T_{32}  \leq K E \int_{0}^t n \int_{\kappa(n,s)}^s |e_r^n|^{p} dr ds
\\
& +  K n^{-p+1} \sum_{k=1}^d E \int_{0}^t  \int_{\kappa(n,s)}^s (1+|x_r|+|x_r^n|)^{\rho p} dr
\\
& \qquad \times \Big|\int_{\kappa(n,s)}^s \sum_{l=1}^m \sum_{i=1}^d \frac{\partial b^k(x_{\kappa(n,r)}^n)}{\partial x^i}  \tilde{\sigma}^{n,(i,l)}(r, x_{\kappa(n,r)}^n)dw_r^l \Big|^p ds
\\
&  + K n^{-\frac{p-2}{2}} \sum_{k=1}^d E \int_{0}^t  \int_{\kappa(n,s)}^s (1+|x_r^n|+|x_{\kappa(n,r)}^n|)^\frac{\rho p}{2} |x_r^n-x_{\kappa(n,r)}^n|^\frac{p}{2} dr
\\
& \qquad \times \Big|\int_{\kappa(n,s)}^s \sum_{l=1}^m \sum_{i=1}^d \frac{\partial b^k(x_{\kappa(n,r)}^n)}{\partial x^i}  \tilde{\sigma}^{n,(i,l)}(r, x_{\kappa(n,r)}^n)dw_r^l \Big|^\frac{p}{2} ds
\\
 & + K n^{-\frac{p-2}{2}} \sum_{k=1}^d E \int_{0}^t \int_{\kappa(n,s)}^s |b(x_{\kappa(n,r)}^n)-b^{n}(x_{\kappa(n,r)}^n) |^\frac{p}{2} dr
\\
& \qquad \times \Big|\int_{\kappa(n,s)}^s \sum_{l=1}^m \sum_{i=1}^d \frac{\partial b^k(x_{\kappa(n,r)}^n)}{\partial x^i}  \tilde{\sigma}^{n,(i,l)}(r, x_{\kappa(n,r)}^n)dw_r^l \Big|^\frac{p}{2} ds
\end{align*}
for any $t \in [0,T]$. Due to H\"older's inequality and an elementary inequality of stochastic integrals, one obtains
\begin{align*}
& T_{32}  \leq K   \int_{0}^t \sup_{0 \leq r \leq s} E|e_r^n|^{p}  ds
\\
& +  K n^{-p+1} \sum_{k=1}^d  \int_{0}^t  \Big[ E \Big(\int_{\kappa(n,s)}^s (1+|x_r|+|x_r^n|)^{\rho p}  dr \Big)^{\frac{p_0}{\rho p}}\Big]^\frac{\rho p}{p_0}
\\
& \qquad \times \Big[E\Big|\int_{\kappa(n,s)}^s \sum_{l=1}^m \sum_{i=1}^d \frac{\partial b^k(x_{\kappa(n,r)}^n)}{\partial x^i}  \tilde{\sigma}^{n,(i,l)}(r, x_{\kappa(n,r)}^n)dw_r^l \Big|^\frac{pp_0}{p_0-\rho p} \Big]^\frac{p_0-\rho p}{p_0} ds
\\
&  + K n^{-\frac{p-2}{2}} \sum_{k=1}^d  \int_{0}^t  \Big[E \Big(\int_{\kappa(n,s)}^s (1+|x_r^n|+|x_{\kappa(n,r)}^n|)^\frac{\rho p}{2} |x_r^n-x_{\kappa(n,r)}^n|^\frac{p}{2} dr \Big)^2\Big]^\frac{1}{2}
\\
& \qquad \times \Big[E\Big|\int_{\kappa(n,s)}^s \sum_{l=1}^m \sum_{i=1}^d \frac{\partial b^k(x_{\kappa(n,r)}^n)}{\partial x^i}  \tilde{\sigma}^{n,(i,l)}(r, x_{\kappa(n,r)}^n)dw_r^l \Big|^p\Big]^\frac{1}{2} ds
 \\
 & + K n^{-\frac{p-2}{2}} \sum_{k=1}^d  \int_{0}^t \Big[E \Big(\int_{\kappa(n,s)}^s |b(x_{\kappa(n,r)}^n)-b^{n}(x_{\kappa(n,r)}^n) |^\frac{p}{2} dr \Big)^2\Big]^\frac{1}{2}
\\
& \qquad \times \Big[E\Big|\int_{\kappa(n,s)}^s \sum_{l=1}^m \sum_{i=1}^d \frac{\partial b^k(x_{\kappa(n,r)}^n)}{\partial x^i}  \tilde{\sigma}^{n,(i,l)}(r, x_{\kappa(n,r)}^n)dw_r^l \Big|^p\Big]^\frac{1}{2} ds
\end{align*}
and this further implies due to H\"older's inequality,
\begin{align*}
&  T_{32} \leq K   \int_{0}^t \sup_{0 \leq r \leq s} E|e_r^n|^{p}  ds
\\
& +  K n^{-p+1}   \int_{0}^t  \Big[ n^{-\frac{p_0}{\rho p}+1}E \int_{\kappa(n,s)}^s (1+|x_r|+|x_r^n|)^{p_0}  dr \Big]^\frac{\rho p}{p_0}
\\
&  \times \Big[n^{-\frac{p p_0}{2 p_0-2\rho p}+1}E\int_{\kappa(n,s)}^s (1+|x_{\kappa(n,r)}^n|)^{\frac{p \rho p_0}{ p_0- \rho p}} | \tilde{\sigma}^n(r,x_{\kappa(n,r)}^n)|^{\frac{pp_0}{ p_0-\rho p}} dr \Big]^{\frac{ p_0- \rho p}{p_0}} ds
\\
&  + K n^{-\frac{p-2}{2}}   \int_{0}^t  \Big[n^{-1}E \int_{\kappa(n,s)}^s (1+|x_r^n|+|x_{\kappa(n,r)}^n|)^{\rho p} |x_r^n-x_{\kappa(n,r)}^n|^p dr \Big]^\frac{1}{2}
\\
& \qquad \times \Big[ n^{-\frac{p}{2}+1} E\int_{\kappa(n,s)}^s (1+|x_{\kappa(n,r)}^n|)^{\rho p} |\tilde{\sigma}^n(r, x_{\kappa(n,r)}^n)|^p dr \Big]^\frac{1}{2} ds
 \\
 & + K n^{-\frac{p-2}{2}}   \int_{0}^t \Big[n^{-1}E\int_{\kappa(n,s)}^s |b(x_{\kappa(n,r)}^n)-b^{n}(x_{\kappa(n,r)}^n) |^p dr \Big]^\frac{1}{2}
\\
& \qquad \times \Big[n^{-\frac{p}{2}+1}E\int_{\kappa(n,s)}^s (1+|x_{\kappa(n,r)}^n|)^{\rho p} |\tilde{\sigma}^n(r, x_{\kappa(n,r)}^n)|^p dr \Big]^\frac{1}{2} ds
\end{align*}
for any $t \in [0,T]$. Again, by using the H\"older's inequality along with Lemmas [\ref{lem:mb:true}, \ref{lem:mbound}, \ref{lem:b-bn:rate}],  one obtains the following estimates,
\begin{align*}
T_{32} & \leq K   \int_{0}^t \sup_{0 \leq r \leq s} E|e_r^n|^{p}  ds
\\
& +  K n^{-p}   \int_{0}^t \Big[n^{-\frac{pp_0}{2p_0-2 \rho p}+1}\int_{\kappa(n,s)}^s \{E(1+|x_{\kappa(n,r)}^n|)^{p_0}\}^\frac{\rho p}{p_0-\rho p}
\\
& \qquad \times \{E|\tilde{\sigma}^n(r, x_{\kappa(n,r)}^n)|^\frac{pp_0}{p_0-2\rho p}\}^{\frac{p_0-2\rho p}{p_0-\rho p}} dr \Big]^\frac{p_0-\rho p}{p_0} ds
\\
&  + K n^{-\frac{p-2}{2}}   \int_{0}^t  \Big[n^{-1} \int_{\kappa(n,s)}^s \{E(1+|x_r^n|+|x_{\kappa(n,r)}^n|)^{p_0}\}^\frac{\rho p}{p_0}
\\
&  \times \{E |x_r^n-x_{\kappa(n,r)}^n|^{\frac{pp_0}{p_0-\rho p}}\}^\frac{p_0-\rho p}{p_0} dr \Big]^\frac{1}{2}
\\
& \quad \times \Big[ n^{-\frac{p}{2}+1} \int_{\kappa(n,s)}^s \{E(1+|x_{\kappa(n,r)}^n|)^{p_0}\}^\frac{\rho p}{p_0}
\\
& \qquad \times \{E|\tilde{\sigma}^n(r, x_{\kappa(n,r)}^n)|^{\frac{pp_0}{p_0-\rho p}}\}^\frac{p_0-\rho p}{ p_0} dr \Big]^\frac{1}{2} ds
 \\
 & + K n^{-p}   \int_{0}^t \Big[n^{-\frac{p}{2}+1}\int_{\kappa(n,s)}^s \{E(1+|x_{\kappa(n,r)}^n|)^{p_0}\}^\frac{\rho p}{p_0}
 \\
 & \qquad \times \{E|\tilde{\sigma}^n(r, x_{\kappa(n,r)}^n)|^{\frac{p_0p}{p_0-\rho p}}\}^\frac{p_0-\rho p}{p_0} dr \Big]^\frac{1}{2} ds
\end{align*}
for any $t \in [0,T]$. Hence, on using Lemmas [\ref{lem:mbound}, \ref{lem:one:step}] and Corollary \ref{cor:tilde:sig}, one obtains,
\begin{align}
T_{32} & \leq  K n^{-p}+K \int_{0}^t  \sup_{0 \leq r \leq s}E|e_r^n|^p  ds\label{eq:T32}
\end{align}
for any $t\in [0,T]$. Further, one observes that the estimation of  $T_{33}$ and $T_{34}$ can be done together as described below. First, one observes that $T_{33}$ can be expressed as
\begin{align*}
T_{33} & :=\sum_{k=1}^d E\int_{0}^t \int_{\kappa(n,s)}^s |e_r^n|^{p-2} \sum_{j=1}^m  \big(\sigma^{(k,j)}(x_r)-\tilde{\sigma}^{n,(k,j)}(r,x_{\kappa(n,r)}^n)\big)dw_r^j
\\
& \qquad \times  \int_{\kappa(n,s)}^s \sum_{l=1}^m \sum_{i=1}^d\frac{\partial b^k(x_{\kappa(n,r)}^n)}{\partial x^i} \tilde{\sigma}^{n,(i,l)}(r, x_{\kappa(n,r)}^n)dw_r^l ds
\\
& = \sum_{k=1}^d E\int_{0}^t \int_{\kappa(n,s)}^s |e_r^n|^{p-2} \sum_{j=1}^m  \big(\sigma^{(k,j)}(x_r)-\tilde{\sigma}^{n,(k,j)}(r, x_{\kappa(n,r)}^n)\big)
\\
& \qquad \times \sum_{i=1}^d\frac{\partial b^k(x_{\kappa(n,r)}^n)}{\partial x^i} \tilde{\sigma}^{n,(i,j)}(r, x_{\kappa(n,r)}^n)dr ds
\end{align*}
which due to Schwartz inequality and Remark \ref{rem:poly:b} yields
\begin{align*}
T_{33} & \leq K E\int_{0}^t \int_{\kappa(n,s)}^s |e_r^n|^{p-2} |\sigma(x_r)-\tilde{\sigma}^{n}(r, x_{\kappa(n,r)}^n)| (1+|x_{\kappa(n,r)}^n|)^{\rho}
\\
& \qquad \times |\tilde{\sigma}^{n}(r, x_{\kappa(n,r)}^n)|dr ds
\end{align*}
for any $t \in [0,T]$. Similarly, $T_{34}$ can be estimated as
\begin{align*}
& T_{34}  :=K \sum_{k=1}^d E\int_{0}^t  \int_{\kappa(n,s)}^s e_r^{n,k} |e_r^n|^{p-4} e_r^n (\sigma(x_r)-\tilde{\sigma}^n(r, x_{\kappa(n,r)}^n))  dw_r
\\
& \qquad \times \int_{\kappa(n,s)}^s \sum_{l=1}^m \sum_{i=1}^d\frac{\partial b^k(x_{\kappa(n,r)}^n)}{\partial x^i}  \tilde{\sigma}^{n,(i,l)}(r, x_{\kappa(n,r)}^n)dw_r^l ds \notag
\\
& =K \sum_{k=1}^d    E\int_{0}^t  \int_{\kappa(n,s)}^s \sum_{j=1}^m \sum_{u=1}^d e_r^{n,k} |e_r^n|^{p-4} e_r^{n,u} (\sigma^{(u,j)}(x_r)-\tilde{\sigma}^{n,(u,j)}(r, x_{\kappa(n,r)}^n))  dw_r^j
\\
& \qquad \times \int_{\kappa(n,s)}^s \sum_{l=1}^m\sum_{i=1}^d \frac{\partial b^k(x_{\kappa(n,r)}^n)}{\partial x^i}  \tilde{\sigma}^{n,(i,l)}(r, x_{\kappa(n,r)}^n)dw_r^l ds \notag
\\
& =K \sum_{k=1}^d    E\int_{0}^t  \int_{\kappa(n,s)}^s \sum_{j=1}^m \sum_{u=1}^d e_r^{n,k} |e_r^n|^{p-4} e_r^{n,u} (\sigma^{(u,j)}(x_r)-\tilde{\sigma}^{n,(u,j)}(r, x_{\kappa(n,r)}^n))
\\
& \qquad \times \sum_{i=1}^d \frac{\partial b^k(x_{\kappa(n,r)}^n)}{\partial x^i}  \tilde{\sigma}^{n,(i,j)}(r, x_{\kappa(n,r)}^n)dr ds \notag
\end{align*}
which on using Remark \ref{rem:poly:b} gives
\begin{align*}
T_{34} & \leq  K E\int_{0}^t \int_{\kappa(n,s)}^s |e_r^n|^{p-2} |\sigma(x_r)-\tilde{\sigma}^{n}(r, x_{\kappa(n,r)}^n)| (1+|x_{\kappa(n,r)}^n|)^{\rho}
\\
& \qquad \times |\tilde{\sigma}^{n}(r, x_{\kappa(n,r)}^n)|dr ds
\end{align*}
for any $t \in [0,T]$. For estimating $T_{33}+T_{34}$, one uses the following splitting,
\begin{align}
\sigma(& x_r)  -\tilde{\sigma}^{n}(r,x_{\kappa(n,r)}^n)=\sigma(x_r) -\sigma^{n}(x_{\kappa(n,r)}^n)- \sigma^{n}_1(r, x_{\kappa(n,r)}^n) \notag
\\
&=(\sigma(x_r)-\sigma(x_r^n)) +(\sigma(x_r^n)-\sigma(x_{\kappa(n,r)}^n)-\sigma_1^{n}(r,x_{\kappa(n,r)}^n)) \notag
\\
& \qquad + (\sigma( x_{\kappa(n,r)}^n)- \sigma^{n}(x_{\kappa(n,r)}^n)) \label{eq:sig:plit}
\end{align}
and hence obtains the following estimates,
\begin{align*}
T_{33}+T_{34}  &\leq  K E\int_{0}^t \int_{\kappa(n,s)}^s |e_r^n|^{p-2}   |\sigma(x_r)-\sigma(x_r^n)| (1+|x_{\kappa(n,r)}^n|)^{\rho}
\\
& \qquad \times |\tilde{\sigma}^{n}(r, x_{\kappa(n,r)}^n)|dr ds
\\
& +K E\int_{0}^t \int_{\kappa(n,s)}^s |e_r^n|^{p-2}  |\sigma(x_r^n)-\sigma(x_{\kappa(n,r)}^n) - \sigma_1^{n}(r, x_{\kappa(n,r)}^n)|
\\
& \qquad \times (1+|x_{\kappa(n,r)}^n|)^{\rho}|\tilde{\sigma}^{n}(r, x_{\kappa(n,r)}^n)|dr ds
\\
&+ K E\int_{0}^t \int_{\kappa(n,s)}^s |e_r^n|^{p-2} |\sigma(x_{\kappa(n,r)}^n)-\sigma^{n}(x_{\kappa(n,r)}^n)|
\\
& \qquad \times (1+|x_{\kappa(n,r)}^n|)^{\rho} |\tilde{\sigma}^{n}(r, x_{\kappa(n,r)}^n)|dr ds
\end{align*}
which also gives the following expressions,
\begin{align*}
T_{33} & +T_{34} \leq  K E\int_{0}^t \int_{\kappa(n,s)}^s (n^\frac{1}{p}|e_r^n|)^{p-1} n^{-\frac{p-1}{p}} (1+ |x_r|+|x_r^n|)^\frac{\rho}{2}
\\
& \qquad \times (1+|x_{\kappa(n,r)}^n|)^{\rho}  |\tilde{\sigma}^{n}(r, x_{\kappa(n,r)}^n)|dr ds
\\
& +K E\int_{0}^t \int_{\kappa(n,s)}^s (n^\frac{1}{p}|e_r^n|)^{p-2}  n^{-\frac{p-2}{p}}|\sigma(x_r^n)-\sigma(x_{\kappa(n,r)}^n) - \sigma_1^{n}(r, x_{\kappa(n,r)}^n)|
\\
& \qquad \times (1+|x_{\kappa(n,r)}^n|)^{\rho}|\tilde{\sigma}^{n}(r, x_{\kappa(n,r)}^n)|dr ds
\\
&+K E\int_{0}^t \int_{\kappa(n,s)}^s (n^\frac{1}{p}|e_r^n|)^{p-2} n^{-\frac{p-2}{p}}|\sigma(x_{\kappa(n,r)}^n)-\sigma^{n}(x_{\kappa(n,r)}^n)|
\\
& \qquad \times (1+|x_{\kappa(n,r)}^n|)^{\rho} |\tilde{\sigma}^{n}(r, x_{\kappa(n,r)}^n)|dr ds
\end{align*}
for any $t \in [0,T]$. Also, on the application of Young's inequality, one obtains
\begin{align*}
T_{33} & +T_{34} \leq  K E\int_{0}^t n\int_{\kappa(n,s)}^s |e_r^n|^p ds
\\
& + K n^{-p+1} E\int_{0}^t \int_{\kappa(n,s)}^s (1+ |x_r|+|x_r^n|)^\frac{\rho p}{2}  (1+|x_{\kappa(n,r)}^n|)^{\rho p}|\tilde{\sigma}^{n}(r, x_{\kappa(n,r)}^n)|^p dr ds
\\
& +K n^{-\frac{p-2}{2}} E\int_{0}^t \int_{\kappa(n,s)}^s   |\sigma(x_r^n)-\sigma(x_{\kappa(n,r)}^n) - \sigma_1^{n}(r, x_{\kappa(n,r)}^n)|^\frac{p}{2}
\\
& \qquad \times (1+|x_{\kappa(n,r)}^n|)^\frac{\rho p}{2}|\tilde{\sigma}^{n}(r, x_{\kappa(n,r)}^n)|^\frac{p}{2} dr ds
\\
&+K n^{-\frac{p-2}{2}} E\int_{0}^t \int_{\kappa(n,s)}^s  |\sigma(r, x_{\kappa(n,r)}^n)-\sigma^{n}(x_{\kappa(n,r)}^n)|^\frac{p}{2}
\\
& \qquad \times (1+|x_{\kappa(n,r)}^n|)^\frac{\rho p}{2} |\tilde{\sigma}^{n}(r,x_{\kappa(n,r)}^n)|^\frac{p}{2} dr ds
\end{align*}
for any $t \in [0,T]$. Moreover, one uses H\"older's inequality to get the following  estimates,
\begin{align*}
& T_{33}  +T_{34} \leq  K E\int_{0}^t \sup_{0 \leq r \leq s} E|e_r^n|^p ds
\\
& + Kn^{-p+1} \int_{0}^t \int_{\kappa(n,s)}^s \{E(1+ |x_r|+|x_r^n|)^{p_0}\}^\frac{p \rho }{2p_0}\{\{E(1+|x_{\kappa(n,r)}^n|)^{p_0}\}^\frac{2\rho p}{2p_0-\rho p}
\\
&\qquad \times \{E|\tilde{\sigma}^n(x_{\kappa(n,r)}^n)|^\frac{2pp_0}{2p_0-3\rho p}\}^\frac{2p_0-3\rho p}{2p_0-\rho p}\}^\frac{2p_0-\rho p}{2p_0}  dr ds
\\
& +K n^{-\frac{p-2}{2}} \int_{0}^t \int_{\kappa(n,s)}^s   \Big\{E|\sigma(x_r^n)-\sigma(x_{\kappa(n,r)}^n) - \sigma_1^{n}(r, x_{\kappa(n,r)}^n)|^{p}\Big\}^\frac{1}{2}
\\
& \qquad \times \Big\{E(1+|x_{\kappa(n,r)}^n|)^{p_0}\}^\frac{\rho p}{p_0} \{E|\tilde{\sigma}^{n}(r, x_{\kappa(n,r)}^n)|^{\frac{pp_0}{p_0-\rho p}}\}^\frac{p_0-\rho p}{p_0}\Big\}^\frac{1}{2} dr ds
\\
&+K n^{-\frac{p-2}{2}} \int_{0}^t \int_{\kappa(n,s)}^s  \Big\{E|\sigma(x_{\kappa(n,r)}^n)-\sigma^{n}(r, x_{\kappa(n,r)}^n)|^p\Big\}^\frac{1}{2}
\\
& \qquad \times \Big\{ \{E(1+|x_{\kappa(n,r)}^n|)^{p_0}\}^\frac{\rho p}{p_0} \{E|\tilde{\sigma}^{n}(r, x_{\kappa(n,r)}^n)|^{\frac{pp_0}{p_0-\rho p}}\}^\frac{p_0-\rho p}{p_0}\Big\}^\frac{1}{2} dr ds
\end{align*}
and then Lemmas [\ref{lem:mbound},  \ref{lem:si-sin:rate}, \ref{lem:sig-sig:rate}] and Corollary \ref{cor:tilde:sig} yield
\begin{align} \label{eq:T33+T34}
T_{33} & +T_{34} \leq K  n^{-p}+ K \int_{0}^t \sup_{0 \leq r \leq s} E |e_r^n|^p ds
\end{align}
for any $t \in [0,T]$. For $T_{35}$, due to \eqref{eq:sig:plit},
\begin{align*}
T_{35}& :=K \sum_{k=1}^d E\int_{0}^t   \int_{\kappa(n,s)}^s |e_r^n|^{p-3}|\sigma(x_r)-\tilde{\sigma}^n(r,x_{\kappa(n,r)}^n)|^2 dr
\\
& \qquad \times \Big|\int_{\kappa(n,s)}^s \sum_{i=1}^d \sum_{l=1}^m \frac{\partial b^k(x_{\kappa(n,r)}^n)}{\partial x^i}  \tilde{\sigma}^{n,(i,l)}(r,x_{\kappa(n,r)}^n)dw_r^l \Big|ds \notag
\\
& \leq K \sum_{k=1}^d E\int_{0}^t   \int_{\kappa(n,s)}^s (n^\frac{1}{p}|e_r^n|)^{p-1}(1+|x_r|+|x_r^n|)^{\rho} dr
\\
&\qquad \times n^{-\frac{p-1}{p}}\Big|\int_{\kappa(n,s)}^s \sum_{l=1}^m \sum_{i=1}^d \frac{\partial b^k(x_{\kappa(n,r)}^n)}{\partial x^i}  \tilde{\sigma}^{n,(i,l)}(r, x_{\kappa(n,r)}^n)dw_r^l \Big|ds \notag
\\
&+ K \sum_{k=1}^d E\int_{0}^t   \int_{\kappa(n,s)}^s (n^\frac{1}{p}|e_r^n|)^{p-3}|\sigma(x_r^n)-\sigma(x_{\kappa(n,r)}^n)- \sigma^n_1(r,x_{\kappa(n,r)}^n)|^2 dr
\\
&\qquad \times n^{-\frac{p-3}{p}}\Big|\int_{\kappa(n,s)}^s \sum_{l=1}^m \sum_{i=1}^d \frac{\partial b^k(x_{\kappa(n,r)}^n)}{\partial x^i}  \tilde{\sigma}^{n,(i,l)}(r, x_{\kappa(n,r)}^n)dw_r^l \Big|ds \notag
\\
&+ K \sum_{k=1}^d E\int_{0}^t   \int_{\kappa(n,s)}^s (n^\frac{1}{p}|e_r^n|)^{p-3}|\sigma(x_{\kappa(n,r)}^n)- \sigma^n(x_{\kappa(n,r)}^n)|^2 dr
\\
&\qquad \times n^{-\frac{p-3}{p}}\Big|\int_{\kappa(n,s)}^s \sum_{l=1}^m \sum_{i=1}^d \frac{\partial b^k(x_{\kappa(n,r)}^n)}{\partial x^i}  \tilde{\sigma}^{n,(i,l)}(r, x_{\kappa(n,r)}^n)dw_r^l \Big|ds \notag
\end{align*}
which on using Young's inequality yields,
\begin{align*}
T_{35} & \leq K  E\int_{0}^t  n \int_{\kappa(n,s)}^s |e_r^n|^{p} ds + K n^{-p+1}\sum_{k=1}^d E\int_{0}^t \int_{\kappa(n,s)}^s (1+|x_r|+|x_r^n|)^{\rho p} dr
\\
&\qquad \times \Big|\int_{\kappa(n,s)}^s \sum_{l=1}^m \sum_{i=1}^d \frac{\partial b^k(x_{\kappa(n,r)}^n)}{\partial x^i}  \tilde{\sigma}^{n,(i,l)}(r, x_{\kappa(n,r)}^n)dw_r^l \Big|^p ds \notag
\\
& + K n^{-\frac{p-3}{3}} \sum_{k=1}^d E\int_{0}^t   \int_{\kappa(n,s)}^s |\sigma(x_r^n)-\sigma(x_{\kappa(n,r)}^n)- \sigma^n_1(r, x_{\kappa(n,r)}^n)|^\frac{2p}{3} dr
\\
&\qquad \times \Big|\int_{\kappa(n,s)}^s \sum_{l=1}^m \sum_{i=1}^d \frac{\partial b^k(x_{\kappa(n,r)}^n)}{\partial x^i}  \tilde{\sigma}^{n,(i,l)}(r, x_{\kappa(n,r)}^n)dw_r^l \Big|^\frac{p}{3} ds \notag
\\
&+ K n^{-\frac{p-3}{3}} \sum_{k=1}^d E\int_{0}^t   \int_{\kappa(n,s)}^s |\sigma(x_{\kappa(n,r)}^n)- \sigma^n(x_{\kappa(n,r)}^n)|^{\frac{2p}{3}} dr
\\
&\qquad \times \Big|\int_{\kappa(n,s)}^s \sum_{l=1}^m \sum_{i=1}^d \frac{\partial b^k(x_{\kappa(n,r)}^n)}{\partial x^i}  \tilde{\sigma}^{n,(i,l)}(r,x_{\kappa(n,r)}^n)dw_r^l \Big|^\frac{p}{3} ds \notag
\end{align*}
and then on applying H\"older's inequality, one obtains
\begin{align*}
& T_{35}  \leq K  \int_{0}^t  \sup_{0 \leq r \leq s} E|e_r^n|^{p} ds
\\
& + K n^{-p+1} \sum_{k=1}^d \int_{0}^t \Big\{ n^{-\frac{p_0}{\rho p}+1}E\int_{\kappa(n,s)}^s (1+|x_r|+|x_r^n|)^{p_0} dr \Big\}^\frac{\rho p}{p_0}
\\
&\qquad \times \Big\{E\Big|\int_{\kappa(n,s)}^s \sum_{l=1}^m \sum_{i=1}^d \frac{\partial b^k(x_{\kappa(n,r)}^n)}{\partial x^i}  \tilde{\sigma}^{n,(i,l)}(r, x_{\kappa(n,r)}^n)dw_r^l \Big|^{\frac{pp_0}{p_0-\rho p}}\Big\}^\frac{p_0-\rho p}{p_0} ds \notag
\\
& + K n^{-\frac{p-3}{3}} \sum_{k=1}^d \int_{0}^t   \Big\{E\Big(\int_{\kappa(n,s)}^s |\sigma(x_r^n)-\sigma(x_{\kappa(n,r)}^n)- \sigma^n_1(r, x_{\kappa(n,r)}^n)|^\frac{2p}{3} dr \Big)^\frac{3}{2}\Big\}^\frac{2}{3}
\\
&\qquad \times \Big\{E\Big|\int_{\kappa(n,s)}^s \sum_{l=1}^m \sum_{i=1}^d \frac{\partial b^k(x_{\kappa(n,r)}^n)}{\partial x^i}  \tilde{\sigma}^{n,(i,l)}(r, x_{\kappa(n,r)}^n)dw_r^l \Big|^{p}\Big\}^\frac{1}{3} ds \notag
\\
&+ K  n^{-\frac{p-3}{3}} \sum_{k=1}^d \int_{0}^t  \Big\{ E\Big(\int_{\kappa(n,s)}^s |\sigma(x_{\kappa(n,r)}^n)- \sigma^n(x_{\kappa(n,r)}^n)|^{\frac{2p}{3}} dr \Big)^\frac{3}{2}\Big\}^\frac{2}{3}
\\
&\qquad \times\Big\{E\Big|\int_{\kappa(n,s)}^s \sum_{l=1}^m \sum_{i=1}^d \frac{\partial b^k(x_{\kappa(n,r)}^n)}{\partial x^i}  \tilde{\sigma}^{n,(i,l)}(r, x_{\kappa(n,r)}^n)dw_r^l \Big|^p\Big\}^\frac{1}{3} ds \notag
\end{align*}
for any $t \in [0,T]$. Further, one uses Remark \ref{rem:poly:b}, an elementary inequality of stochastic integrals and H\"older's inequality to obtain the following estimates,
\begin{align*}
&  T_{35} \leq K  \int_{0}^t  \sup_{0 \leq r \leq s} E|e_r^n|^{p} ds + K  n^{-p} \int_{0}^t \Big\{n^{-\frac{pp_0}{2p_0-2\rho p}+1}
\\
& \qquad \times E\int_{\kappa(n,s)}^s (1+|x_{\kappa(n,r)}^n|)^{\frac{p \rho p_0}{p_0-\rho p}}  |\tilde{\sigma}^{n}(r,x_{\kappa(n,r)}^n)|^{\frac{pp_0}{p_0-\rho p}} dr \Big\}^\frac{p_0-\rho p}{p_0} ds \notag
\\
& + K n^{-\frac{p-3}{3}} \int_{0}^t   \Big\{n^{-\frac{1}{2}}E \int_{\kappa(n,s)}^s |\sigma(x_r^n)-\sigma(x_{\kappa(n,r)}^n)- \sigma^n_1(r,x_{\kappa(n,r)}^n)|^p dr \Big\}^\frac{2}{3}
\\
&\qquad \times \Big\{n^{-\frac{p}{2}+1}E\int_{\kappa(n,s)}^s (1+|x_{\kappa(n,r)}^n|)^{\rho p}  |\tilde{\sigma}^{n}(r,x_{\kappa(n,r)}^n)|^p dr \Big\}^\frac{1}{3} ds \notag
\\
&+ K n^{-\frac{p-3}{3}} \int_{0}^t  \Big\{ n^{-\frac{1}{2}}E\int_{\kappa(n,s)}^s |\sigma(x_{\kappa(n,r)}^n)- \sigma^n(x_{\kappa(n,r)}^n)|^{p} dr \Big\}^\frac{2}{3}
\\
&\qquad \times \Big\{n^{-\frac{p}{2}+1}E\int_{\kappa(n,s)}^s (1+|x_{\kappa(n,r)}^n|)^{\rho p}  |\tilde{\sigma}^{n}(r, x_{\kappa(n,r)}^n)|^p dr \Big\}^\frac{1}{3} ds \notag
\end{align*}
which due to further application of Young's inequality gives,
\begin{align*}
& T_{35}  \leq K  \int_{0}^t  \sup_{0 \leq r \leq s} E|e_r^n|^{p} ds+ K  n^{-p} \int_{0}^t \Big\{n^{-\frac{pp_0}{2p_0-2\rho p}+1}
\\
& \times \int_{\kappa(n,s)}^s \{E(1+|x_{\kappa(n,r)}^n|)^{p_0}\}^{\frac{\rho p}{p_0-\rho p}}  \{E|\tilde{\sigma}^{n}(r,x_{\kappa(n,r)}^n)|^{\frac{pp_0}{p_0-2\rho p}}\}^\frac{p_0-2\rho p}{p_0-\rho p} dr \Big\}^\frac{p_0-\rho p}{p_0} ds \notag
\\
& + K n^{-\frac{p-3}{3}} \int_{0}^t   \Big\{n^{-\frac{1}{2}}E \int_{\kappa(n,s)}^s |\sigma(x_r^n)-\sigma(x_{\kappa(n,r)}^n)- \sigma^n_1(r,x_{\kappa(n,r)}^n)|^p dr \Big\}^\frac{2}{3}
\\
& \times \Big\{n^{-\frac{p}{2}+1}\int_{\kappa(n,s)}^s \{E(1+|x_{\kappa(n,r)}^n|)^{p_0}\}^\frac{\rho p}{p_0}  \{E|\tilde{\sigma}^{n}(r,x_{\kappa(n,r)}^n)|^{\frac{p_0p}{p_0-\rho p}}\}^\frac{p_0-\rho p}{p_0} dr \Big\}^\frac{1}{3} ds \notag
\\
&+ K n^{-\frac{p-3}{3}} \int_{0}^t  \Big\{ n^{-\frac{1}{2}}E\int_{\kappa(n,s)}^s |\sigma(x_{\kappa(n,r)}^n)- \sigma^n(x_{\kappa(n,r)}^n)|^{p} dr \Big\}^\frac{2}{3}
\\
&\times \Big\{n^{-\frac{p}{2}+1}E\int_{\kappa(n,s)}^s \{E(1+|x_{\kappa(n,r)}^n|)^{p_0}\}^\frac{\rho p}{p_0}  \{E|\tilde{\sigma}^{n}(r,x_{\kappa(n,r)}^n)|^{\frac{p_0p}{p_0-\rho p}}\}^\frac{p_0-\rho p}{p_0} dr \Big\}^\frac{1}{3} ds \notag
\end{align*}
and finally on the application of Lemmas [\ref{lem:mbound}, \ref{lem:si-sin:rate}, \ref{lem:sig-sig:rate}] and  Corollary \ref{cor:tilde:sig}, one obtains
\begin{align} \label{eq:T35}
T_{35} \leq  K n^{-p} + K\int_{0}^t  \sup_{0 \leq r \leq s}E|e_r^n|^p   ds
\end{align}
for any $t \in [0,T]$. Hence, on substituting estimates from \eqref{eq:T31}, \eqref{eq:T32}, \eqref{eq:T33+T34} and \eqref{eq:T35} in \eqref{eq:T31+T35}, one obtains
\begin{align} \label{eq:T3}
T_{3} \leq    K n^{-p}+K\int_{0}^t  \sup_{0 \leq r \leq s}E|e_r^n|^p   ds
\end{align}
for any $t \in [0,T]$. Thus, the proof is completed by combining estimates from \eqref{eq:T1}, \eqref{eq:T2} and \eqref{eq:T3} in \eqref{eq:T1+T2+T3}.
\end{proof}
%-------------------proof of main theorem-----------------------
\begin{proof}[\bf  Proof of Theorem \ref{thm:main}]
%----------------------------------------------------------------
Let $\bar{b}^n(s):=b(x_s)-b^n(x_{\kappa(n,s)}^n)$ and $\bar{\sigma}^n(s):=\sigma(x_s)-\tilde{\sigma}^n(x_{\kappa(n,s)}^n)$ and then one writes
\begin{align*}
e_t^n:=x_t-x_t^n= \int_0^t \bar{b}^n(s) ds + \int_0^t \bar{\sigma}^n(s)dw_s
\end{align*}
for any $t \in [0,T]$. By the application of It\^o's formula,
\begin{align*}
|e_t^n|^p & =p \int_0^t |e_s^n|^{p-2} e_s^n \bar b^n(s)ds+p \int_0^t |e_s^n|^{p-2} e_s^n \bar \sigma^n(s)dw_s
\\
& +\frac{p(p-2)}{2} \int_0^t |e_s^n|^{p-4} |\bar \sigma^{n*}(s)e_s^n|^2 ds + \frac{p}{2} \int_0^t |e_s^n|^{p-2} |\bar \sigma^{n}(s)|^2 ds
\end{align*}
for any $t \in [0,T]$. As before, when $p=2$ the third term does appear on the right hand side of the above equation and $|e_s^n|^{p-2}$ is absent from the rest of the terms.  Due to Cauchy-Bunyakovsky-Schwartz inequality, one obtains
\begin{align*}
E|e_t^n|^p & \leq  p E\int_0^t |e_s^n|^{p-2} e_s^n \bar b^n(s)ds +\frac{p(p-1)}{2} E \int_0^t |e_s^n|^{p-2} |\bar \sigma^{n}(s)|^2 ds
\end{align*}
for any $t \in [0,T]$. Furthermore, one observes that for $z_1,z_2 \in \mathbb{R}^{d \times m}$, $|z_1+z_2|^2 = |z_1|^2+2\sum_{i=1}^{d}\sum_{j=1}^{m} z_1^{(i,j)} z_2^{(i,j)}+|z_2|^2$, which  on using Young's inequality further implies  $|z_1+z_2|^2 \leq (1+\epsilon)|z_1|^2+(1+1/\epsilon)|z_2|^2$ for every $\epsilon>0$. Let us now fix $\epsilon>0$. Hence, one can use this arguments for estimating $|\sigma(x_s)-\sigma(x_s^n)|^2$ when using the splitting given in equation \eqref{eq:sig:plit}. This along with the splitting of equation \eqref{eq:b:split} gives
\begin{align*}
E|e_t^n|^p & \leq p E\int_0^t |e_s^n|^{p-2} e_s^n\{b(x_s)-b(x_s^n)\}ds
\\
& + p E\int_0^t |e_s^n|^{p-2} e_s^n \{b(x_s^n)-b(x_{\kappa(n,s)}^n)\} ds
\\
& + p E\int_0^t |e_s^n|^{p-2} e_s^n \{b(x_{\kappa(n,s)}^n)-b^n(x_{\kappa(n,s)}^n)\} ds
\\
& +\frac{\epsilon+1}{2}p(p-1) E\int_0^t |e_s^n|^{p-2} |\sigma(x_s)-\sigma(x_s^n)|^2 ds
\\
& + K E\int_0^t |e_s^n|^{p-2} |\sigma(x_s^n)-\sigma(x_{\kappa(n,s)}^n)-\sigma^{n}_1(x_{\kappa(n,s)}^n)|^2 ds
\\
& + K E\int_0^t |e_s^n|^{p-2} |\sigma(x_{\kappa(n,s)}^n)-\sigma^n(x_{\kappa(n,s)}^n)|^2 ds
\end{align*}
for any $t \in [0,T]$. Notice that the constant $K>0$ (a large constant) in the last two terms of the above inequality depends on $\epsilon$. Also, one obtains the following estimates,
\begin{align}
E|e_t^n|^p & \leq \frac{p}{2} E\int_0^t |e_s^n|^{p-2} \Big[e_s^n\{b(x_s)-b(x_s^n)\}+ (1+\epsilon)(p-1) |\sigma(x_s)-\sigma(x_s^n)|^2 \Big] ds \notag
\\
& + p E\int_0^t |e_s^n|^{p-2} e_s^n \{b(x_s^n)-b(x_{\kappa(n,s)}^n)\} ds \notag
\\
& + p E\int_0^t |e_s^n|^{p-2} e_s^n \{b(x_{\kappa(n,s)}^n)-b^n(x_{\kappa(n,s)}^n)\} ds \notag
\\
& + K E\int_0^t |e_s^n|^{p-2} |\sigma(x_s^n)-\sigma(x_{\kappa(n,s)}^n)-\sigma^{n}_1(x_{\kappa(n,s)}^n)|^2 ds \notag
\\
& + K E\int_0^t |e_s^n|^{p-2} |\sigma(x_{\kappa(n,s)}^n)-\sigma^n(x_{\kappa(n,s)}^n)|^2 ds. \notag
\end{align}
for any $t \in [0,T]$. Since $p < p_1$, thus on using Assumption A-\ref{as:sde:lipschitz}, Lemmas [\ref{lem:sig-sig:rate}, \ref{lem:a-tilde a:rate:new}] and Young's inequality, one obtains
\begin{align}
E|e_t^n|^p & \leq K E\int_0^t |e_s^n|^{p}  ds  + Kn^{-p}  + K E\int_0^t |b(x_{\kappa(n,s)}^n)-b^n(x_{\kappa(n,s)}^n)|^p ds \notag
\\
& \quad+ K E\int_0^t  |\sigma(x_{\kappa(n,s)}^n)-\sigma^n(x_{\kappa(n,s)}^n)|^p ds \notag
\end{align}
and hence  Lemmas [\ref{lem:b-bn:rate}, \ref{lem:si-sin:rate}] give
\begin{align}
\sup_{0 \leq s \leq t}E|e_s^n|^p & \leq K \int_0^t \sup_{0 \leq r \leq s}E|e_r^n|^{p}  ds  + Kn^{-p} <\infty
\end{align}
for any $t \in [0,T]$. Finally, the application of Gronwall's lemma completes the proof.
\end{proof}
%-----------------------------------------------------------------

\end{document}